\documentclass[11pt]{article}
\usepackage[utf8]{inputenc}
\usepackage[T1]{fontenc}
\usepackage{amsmath, amssymb, amsthm}
\usepackage[margin=1in]{geometry}
\usepackage{xcolor}
\newcommand{\rev}[1]{#1}
\usepackage{graphicx}
\graphicspath{{./}{reda/}}
\usepackage{booktabs}
\usepackage[hidelinks]{hyperref}
\usepackage{url}
\usepackage{doi}
\usepackage{comment}
\newtheorem{remark}{Remark}
\newtheorem{lemma}{Lemma}
\newtheorem{theorem}{Theorem}
\newtheorem{proposition}{Proposition}
\numberwithin{equation}{section}
\newtheoremstyle{examplestyle}
  {6pt}
  {6pt}
  {\normalfont}
  {}
  {\normalfont\itshape}
  {.}
  {0.5em}
  {}
\theoremstyle{examplestyle}
\newtheorem{example}{Example}[subsection]
\theoremstyle{plain}
\usepackage{cite}
\usepackage{float}
\usepackage{algorithm}
\usepackage{algpseudocode}
\usepackage{authblk}

\begin{document}

\title{
A cylindrical neural approximation theorem for conditional laws of McKean–Vlasov equations with common noise
}

\author{%
Nacira Agram\textsuperscript{1,2}
\hspace{1.5cm}
Reda Hmioui\textsuperscript{1,3}
\hspace{1.5cm}
Jan Rems\textsuperscript{4}
}

\date{August 8, 2026}

\maketitle

\footnotetext[1]{Department of Mathematics, KTH Royal Institute of Technology, 100 44 Stockholm, Sweden.
Email: nacira@kth.se. Supported by the Swedish Research Council grant (2020-04697).}

\footnotetext[2]{Digital Futures, Sweden.}

\footnotetext[3]{Grenoble INP -- Ensimag, Université Grenoble Alpes,
Grenoble, France.
Email: reda.hmioui@grenoble-inp.org.}

\footnotetext[4]{Department of Mathematics, University of Ljubljana,
Ljubljana, Slovenia.
Email: jan.rems@fmf.uni-lj.si.  Supported by the Slovenian Research and Innovation Agency, research core funding No. P1-0448.}

\begin{abstract}
We introduce conditional cylindrical neural networks for approximating
functionals of conditional laws in McKean--Vlasov equations with common
noise. Fourier moments of the initial law and truncated signatures of the
\rev{time augmented} common noise are mapped by a \rev{mixture density} network to a
Gaussian mixture approximation of the conditional law. A cylindrical neural
network then evaluates the target functional through analytic integrals
against this predicted measure.

Rough path \rev{well posedness} and stability provide a conditional law map that is
continuous in the initial distribution and the rough driver and agrees almost
surely with the classical conditional law at the It\^o Brownian lift.
Combining this continuity with Fourier separation, signature uniqueness,
Wasserstein density of Gaussian mixtures, and neural universal approximation,
we prove an $L^2$ universal approximation theorem for continuous square
integrable functionals.

The numerical study implements the resulting two stage procedure on six
examples, including \rev{non Gaussian} initial laws, nonlinear drift, multiplicative
common noise, and a \rev{two dimensional} state. Independent particle references are
used when no \rev{closed form} law is available. The learned conditional law and
functional approximations consistently improve on the empirical particle plug
in, and additional experiments examine feature sensitivity, training from one
\rev{terminal observation per common noise scenario}, and It\^o--Stratonovich
consistency.
\end{abstract}

\rev{mixture density} networks; universal approximation; deep learning.


\section{Introduction}
McKean--Vlasov dynamics with common noise describe large populations of interacting agents driven simultaneously by idiosyncratic randomness and by an aggregate shock; they arise as limits of interacting particle systems and as the forward dynamics of \rev{mean field} games in a random environment, see Carmona and Delarue \cite{carmona_delarue_2018}. \rev{A numerical treatment of conditional McKean--Vlasov dynamics with common noise is given by Agram and Rems~\cite{agram2025conditional}.} In the presence of common noise, the natural state variable is no longer the law of a representative agent but its \emph{conditional} law given the common noise, \(\mu_t=\mathcal L(X_t\mid\mathcal F^{W^0}_t)\) --- a \rev{measure valued} stochastic process. Many quantities of interest, such as value functions or risk and pricing functionals, take the form \(V(x,\mu_t)\) and must be evaluated along this random flow of measures. This raises a basic question of approximation theory: can such functionals be represented, to arbitrary accuracy, by a single trainable architecture whose inputs are the initial law and the \rev{common noise path}?
 
Two obstructions stand in the way. First, the conditional law is only defined almost surely: classical \rev{well posedness} theory provides no topology on the \rev{common noise path} for which \(\omega^0\mapsto\mu_t(\omega^0)\) is a continuous function, and without continuity on a Polish space the \rev{compactness based} approximation arguments below could not even be formulated. We resolve this through the rough path formulation of Friz, Hocquet and L\^e \cite{friz2025mckean}: freezing the common noise as a deterministic rough path, their \rev{well posedness} and stability theory produces a conditional law map \(\rev{\Phi_{T,t}}\), jointly continuous in the initial law and in the driver; their randomization procedure, combined with the
identification theorem of Friz, Lê and Zhang~\cite[Theorem~3.5]{friz2025randomisation}, shows that $\rev{\Phi_{T,t}}$ evaluated at the Brownian rough path lift coincides almost surely with
the conditional law. All of this is recorded in Proposition~1, for which we
claim no originality. Because the analytic theory of~\cite{friz2025mckean} does not require the driver to be geometric, we drive the equation by the Itô lift of the
\rev{time augmented} common noise, for which the identification holds without any
correction term; geometricity, which is required by the signature uniqueness
theorem~\cite{boedihardjo2014signature}, enters only through the \emph{features}, computed from the Stratonovich lift, itself obtained from the Itô lift by an explicit deterministic \rev{level two} translation (Section~2.1).
 
Second, \(V\) is a function on \(\mathbb R^d\times\mathcal P_2(\mathbb R^d)\), an \rev{infinite dimensional} input. Following Pham and Warin \cite{pham2022meanfield}, the measure argument is processed through a cylindrical network \(\Psi(x,\langle\varphi,\mu\rangle)\); the conditional structure is captured by a \rev{mixture density} network $\hat N$ with Gaussian mixture output, fed by finitely many features of the data: Fourier moments $E_L(\mu_0)$ of the initial law and the truncated signature $S_M$ of the lifted common noise, restricted to $[0,t]$. Our main result (Theorem~1) states that the class
\[
\hat V(x,\mu_0,w)=\Psi\!\left(x,\langle \varphi,\hat N(E_L(\mu_0),S_M(r_t w))\rangle\right)
\]
is dense in $L^2(\nu\otimes \mathbf P^0\otimes \mu_t)$, where $r_t$ denotes the restriction of the rough path to the interval $[0,t]$.

The factorization separates two tasks. The \rev{mixture density} network provides a
differentiably parametrized \rev{finite dimensional} surrogate of the pathwise
conditional law, while the cylindrical network evaluates measure dependent objectives without
new particle simulation at inference time. Algorithm~\ref{alg:conditional-cylindrical-training}
turns this representation into a two stage procedure: the law network is first
trained by conditional likelihood, then frozen while the cylindrical network
is trained by regression.

The present paper is an \rev{approximation theoretic} complement to the analytic theory of~\cite{friz2025mckean,friz2025randomisation}. Proposition~1 assembles their \rev{well posedness}, stability and randomization results into the continuity statement required here. The remaining imported ingredients are signature uniqueness~\cite{boedihardjo2014signature}, cylindrical density in the spirit of Pham--Warin~\cite{pham2022meanfield}, $\mathcal W_2$ density of Gaussian mixtures and the classical universal approximation theorem~\cite{hornik1991approximation}. What is new is their assembly into the \rev{finite factorization} Lemma~1 and Theorem~1: on compact sets, the conditional law map can be approximated through finitely many Fourier moments of the initial law and signature coordinates of the common noise, and the resulting \rev{law valued} factor can be realized by an MDN.

The theorem remains qualitative and gives no rates in the Fourier order,
signature level, number of mixture components, or network size. The numerical
study therefore examines the representation empirically rather than claiming a
convergence rate. Six numbered examples cover Gaussian, uniform, and Gaussian
mixture initial laws, a \rev{double well} model, multiplicative common noise, and a
\rev{two dimensional} system. For models without a \rev{closed form} conditional law,
independent particle systems provide the measure inputs and regression targets,
which avoids direct information reuse between the two channels. The study
reports results across multiple seeds, optimization diagnostics, sensitivity to
features and architecture, training from one \rev{terminal observation per
common noise scenario}, and an
It\^o--Stratonovich consistency check. Complete results and aggregation scripts
are provided in the companion repository.

Signature encodings of the common noise have recently been used in numerical schemes for conditional McKean--Vlasov dynamics \cite{agram2025conditional}; the present work complements these algorithmic contributions with an approximation guarantee, the continuity input being supplied by \cite{friz2025mckean}. The flow of conditional measures can alternatively be characterized as the solution of a nonlinear rough Fokker--Planck equation \cite{bugini2025nonlinear}, a natural route towards quantitative refinements of Theorem 1.
 
The paper is organized as follows. Section~2 introduces the functional setting,
states the standing assumptions, proves Proposition~1, and defines the network
classes. Section~3 states and proves the universal approximation theorem.
Section~4 reports the numerical experiments.

\section{Setting and notation}

This section fixes the probabilistic, rough path, and approximation notation
used throughout the paper. We first introduce the spaces of probability
measures and \rev{common noise paths}, then define path restriction and truncated
signature features. We finally state the pathwise conditional law map and
describe the conditional cylindrical network class used in the approximation
theorem.

\subsection{Spaces}
We equip \(\mathcal P_2(\mathbb R^d)\) with the Wasserstein distance
\(\mathcal W_2\). Then \((\mathcal P_2(\mathbb R^d),\mathcal W_2)\) is a Polish space. More generally, \((\mathcal{P}_m(\mathbb{R}^d), \mathcal{W}_m)\) is a complete
separable, hence Polish, metric space for every \(m \geq 2\)
\cite[Theorem~6.18]{villani_ot}; this will be needed whenever the
rough path \rev{well posedness} theory imposes a moment exponent \(m>2\).

Fix a horizon \(T>0\), an integer \(q\geq1\), and let
$
\Omega^0
=
C_0([0,T];\mathbb R^q)
=
\{\omega^0\in C([0,T];\mathbb R^q):\omega^0_0=0\},
$
endowed with the Wiener measure \(\mathbb P^0\) and its canonical filtration
$
\mathbb F^0=(\mathcal F^0_t)_{t\le T}.
$

Fix \(\alpha\in(1/3,1/2)\). Since Brownian paths have infinite
\(1\) variation, the iterated Riemann--Stieltjes integrals defining the
signature are not well defined pathwise. We therefore work with the \rev{time augmented} common noise
$
\widehat W^0_t := (t,W^0_t)\in\mathbb R^{q+1}.
$
We write \(C^{0,\alpha}([0,T];\mathbb R^{q+1})\) for the space of
step two \(\alpha\) Hölder rough paths, not necessarily geometric, and
\(C_g^{0,\alpha}([0,T];\mathbb R^{q+1})\) for its geometric
subspace. Both are equipped with the rough path metric \(\rho_\alpha\).
The geometric space is the closure of the canonical lifts of smooth paths;
it is Polish, and the Stratonovich Brownian lift belongs to it almost surely
\cite[Chapters~8--9 and~13]{friz_victoir_roughpaths}.

The \rev{time augmented} Brownian path has two lifts with different roles. For
$0\leq s\leq t\leq T$, define their second levels by
\[
\widehat{\mathbb W}^{0}_{s,t}
:=
\int_s^t
\bigl(\widehat W_u^0-\widehat W_s^0\bigr)
\otimes\circ d\widehat W_u^0,
\]
and
\[
\widehat{\mathbb W}^{0,\mathrm{It\hat{o}}}_{s,t}
:=
\int_s^t
\bigl(\widehat W_u^0-\widehat W_s^0\bigr)
\otimes d\widehat W_u^0.
\]
Here $\otimes$ denotes the tensor product, while $\circ d\widehat W_u^0$
indicates Stratonovich integration. The Stratonovich lift is
\[
\widehat{\mathbf W}^{0}
=\bigl(\widehat W^0,\widehat{\mathbb W}^{0}\bigr)
\in C_g^{0,\alpha}([0,T];\mathbb R^{q+1})
\qquad \mathbb P^0\text{ a.s.}
\]
and supplies the signature features. Its It\^o lift is
\[
\widehat{W}^{0,\mathrm{It\hat{o}}}
=\bigl(\widehat W^0,\widehat{\mathbb W}^{0,\mathrm{It\hat{o}}}\bigr)
\in C^{0,\alpha}([0,T];\mathbb R^{q+1})
\qquad \mathbb P^0\text{ a.s.}
\]
and supplies the driver used in the analytic identification of the
conditional law.

The two lifts have the same first level and differ at the second level by the
deterministic It\^o--Stratonovich translation displayed in Remark~2. This
relation is continuous in both directions, so the two lifts contain the same
pathwise information, even though only the Stratonovich lift is geometric.

We keep distinct notation for the two probability measures used below:
\(\mathbb P^0\) is Wiener measure on the raw path space \(\Omega^0\), while
\[
\mathbf P^0
:=\bigl(\widehat{\mathbf W}^{0}\bigr)_{\#}\mathbb P^0
\]
is the law of the geometric lifted path on
\(C_g^{0,\alpha}([0,T];\mathbb R^{q+1})\). Integrals over rough
paths use \(\mathbf P^0(d\mathbf w)\); almost sure statements on the raw
Wiener space use \(\mathbb P^0\).

Finally, the time augmentation removes the tree like ambiguity in signature
uniqueness. By the uniqueness theorem of Boedihardjo--Geng--Lyons--Yang
\cite[Theorem~1.1]{boedihardjo2014signature}, the signature is injective on
the \rev{time augmented} geometric rough paths considered here.

\subsection{Restriction to $[0,t]$ and truncated signature}

Throughout the paper, a fixed evaluation time $t \in (0,T]$ is considered; the case $t=0$ is trivial, since $\mu_0$ is part of the data. Because the value $\mu_t$ can only depend on the driver through its restriction to $[0,t]$, we introduce the restriction operator first and define all path features on $[0,t]$.

\paragraph{Restriction to $[0,t]$.}

For $t\in(0,T]$, let
\[
r_t:C^{0,\alpha}_g([0,T];\mathbb{R}^{q+1}) \to C^{0,\alpha}_g([0,t];\mathbb{R}^{q+1}),
\qquad
r_t w := w_{|[0,t]},
\]
where both levels of the rough path are restricted to $\{(s,u):0\le s\le u\le t\}$. Since the Hölder seminorms on a subinterval are bounded by those on the whole interval, $r_t$ is continuous, in fact 1 Lipschitz, for the rough path metric $\rho_\alpha$. Consequently, if $K_\omega\subset C^{0,\alpha}_g([0,T];\mathbb{R}^{q+1})$ is compact, then $r_t(K_\omega)$ is compact in $C^{0,\alpha}_g([0,t];\mathbb{R}^{q+1})$.

\paragraph{Truncated signature.}

For $v\in C^{0,\alpha}_g([0,t];\mathbb{R}^{q+1})$ and $M\in\mathbb{N}$, write
\[
S_M(v)=\bigl(1,v^{(1)}_{0,t},\ldots,v^{(M)}_{0,t}\bigr)\in T^{(M)}(\mathbb{R}^{q+1}),
\]
the signature truncated at level $M$, valued in the truncated tensor algebra of dimension
\[
d_{\mathrm{sig}}(M,q)=\sum_{j=0}^{M}(q+1)^j=\frac{(q+1)^{M+1}-1}{q}.
\]

The map $v\mapsto S_M(v)$ is continuous from
$\bigl(C^{0,\alpha}_g([0,t];\mathbb{R}^{q+1}),\rho_\alpha\bigr)$
to $T^{(M)}(\mathbb{R}^{q+1})$: the signature map is continuous in the
$p$ variation rough path topology \cite[Section~9.1.3]{friz_victoir_roughpaths},
and the $\alpha$ Hölder rough path topology is stronger than the
$p$ variation topology for $p=1/\alpha$. In particular, the composition
$w\mapsto S_M(r_t w)$ is continuous on
$C^{0,\alpha}_g([0,T];\mathbb{R}^{q+1})$. Every signature feature
below is computed from the restricted geometric path and therefore appears
as $S_M(r_t w)$.

\subsection{The target function: Exact conditional law}
Fix $t\in(0,T]$. Proposition~1 below provides, under the assumptions of
Friz--Hocquet--L\^e~\cite{friz2025mckean}, an analytic law map
\[
\rev{\Phi_{T,t}} : \mathcal{P}_m(\mathbb{R}^d)\times
C^{0,\alpha}([0,T];\mathbb{R}^{q+1})
\longrightarrow \mathcal{P}_2(\mathbb{R}^d),
\]
together with its restricted version
\[
\overline{\Phi}_t : \mathcal{P}_m(\mathbb{R}^d)\times
C^{0,\alpha}([0,t];\mathbb{R}^{q+1})
\longrightarrow \mathcal{P}_2(\mathbb{R}^d),
\]
related by the causality identity
\[
\rev{\Phi_{T,t}}(\mu_0,w)
=\overline{\Phi}_t(\mu_0,r_t w).
\]
Both maps are continuous for $\mathcal W_m\times\rho_\alpha$ at the source
and $\mathcal W_2$ at the target. When the driver is the It\^o lift, the
full horizon map $\rev{\Phi_{T,t}}$ is identified almost surely with
$\mathcal L(X_t\mid\mathcal F_t^{W^0})$.

Only the continuity of the restricted map $\overline{\Phi}_t$ is used in the
approximation theory of Section~3. Accordingly, the driver $\mathbf w$ ranges
over geometric lifts so that signature features can be used, and we write
\[
\mu_t(\mathbf w):=\overline{\Phi}_t(\mu_0,r_t\mathbf w),
\]
suppressing the dependence on $\mu_0$ from the notation. By the identification
of Proposition~1 and the \rev{level two} relation discussed in Remark~2, this map
coincides almost surely with the conditional law of the It\^o system carried
by the same Brownian path.

Let $\nu$ be a training measure on $\mathcal{P}_m(\mathbb{R}^d)$ (the law of the initial condition $\mu_0$). The \rev{mean field} function to be learned is
\[
V:\mathbb{R}^d\times \mathcal{P}_2(\mathbb{R}^d)\to \mathbb{R}^p,
\]
assumed continuous for the product topology on $\mathbb{R}^d\times \mathcal{P}_2(\mathbb{R}^d)$, where $\mathbb{R}^d$ carries its Euclidean topology and $\mathcal{P}_2(\mathbb{R}^d)$ the topology induced by $W_2$. This is the hypothesis of Pham--Warin. We further assume the mean square integrability condition
\[
\|V\|_{\Gamma_2}^2
:=\int_{\mathcal{P}_m(\mathbb{R}^d)}
\int_{C_g^{0,\alpha}([0,T];\mathbb{R}^{q+1})}
\int_{\mathbb{R}^d}
\bigl|V(x,\mu_t(\mathbf w))\bigr|^2\,
\mu_t(\mathbf w)(dx)\,\mathbf P^0(d\mathbf w)\,\nu(d\mu_0)<\infty,
\]
where $\mu_t(\mathbf w)=\overline{\Phi}_t(\mu_0,r_t\mathbf w)$ as above.

\subsection{The conditional law map}

The conditional law of a McKean--Vlasov equation with common noise is usually
defined only almost surely as a random probability measure. For the
approximation argument, we need a pathwise version that can be evaluated at a
deterministic realization of the common noise rough path and that depends
continuously on both the initial law and the driver. The following proposition
records this consequence of the \rev{well posedness}, stability, and randomization
results of Friz--Hocquet--L\^e and Friz--L\^e--Zhang
\cite{friz2025mckean,friz2025randomisation}. It provides the analytic
input used in Sections~3--4; no originality is claimed for this result.

Fix $\alpha\in(1/3,1/2)$ and a moment exponent $m\ge2$ for which the
\rev{well posedness} and stability results of Friz--Hocquet--L\^e
\cite[Theorems~3.10--3.11]{friz2025mckean} apply. Let
$a=(a_s)_{s\leq T}$ be a fixed admissible control with values in a control
space $\mathsf A$. For some $d_B\geq1$, let $B$ be an
$\mathbb R^{d_B}$ valued idiosyncratic Brownian motion, and let
$w\in C^{0,\alpha}([0,T];\mathbb R^{q+1})$ be a deterministic rough driver,
not necessarily geometric. For an initial law
$\mu_0\in\mathcal P_m(\mathbb R^d)$, consider the frozen McKean--Vlasov
rough equation

\[
\begin{aligned}
b&:\mathbb R^d\times\mathcal P_m(\mathbb R^d)\times\mathsf A
\longrightarrow\mathbb R^d,\\
\sigma&:\mathbb R^d\times\mathcal P_m(\mathbb R^d)\times\mathsf A
\longrightarrow
\mathcal L(\mathbb R^{d_B},\mathbb R^d)
\cong\mathbb R^{d\times d_B},\\
f&:\mathbb R^d\times\mathcal P_m(\mathbb R^d)\times\mathsf A
\longrightarrow
\mathcal L(\mathbb R^{q+1},\mathbb R^d)
\cong\mathbb R^{d\times(q+1)}.
\end{aligned}
\]

With these coefficient maps, the equation is
\[
dX_s=b(X_s,\mu_s,a_s)\,ds
+\sigma(X_s,\mu_s,a_s)\,dB_s
+f(X_s,\mu_s,a_s)\,dw_s,
\qquad
\mu_s=\mathcal{L}(X_s),
\qquad
\mathcal L(X_0)=\mu_0.
\]
Assume that the coefficients are bounded and satisfy the regularity
hypotheses of \cite[Theorems~3.10--3.11]{friz2025mckean}. In the cited
controlled path formulation, the rough coefficient is specified by a pair
$(f,f')$. When the coefficients have no explicit dependence on the rough
driver, the direct driver derivative may be taken as $f'=0$. This does not
mean that the controlled integrand $f(X,\mu,a)$ has zero Gubinelli
derivative: its derivative still contains the spatial and measure terms
produced by composition with the solution
\cite[Definition~3.7 and Eq.~(3.13)]{friz2025mckean}. The first column of
$f\in\mathbb R^{d\times(q+1)}$ multiplies the time component of the augmented
driver; in the applications below it is zero because the drift is kept in
$b$.

\begin{proposition}[Pathwise conditional law map]
Under the preceding assumptions, the frozen equation has a unique solution

$X^{\mu_0,w}$. For every $t\in[0,T]$, the map
\[
\rev{\Phi_{T,t}}:(\mathcal P_m(\mathbb R^d),\mathcal W_m)
\times\bigl(C^{0,\alpha}([0,T];\mathbb R^{q+1}),\rho_\alpha\bigr)
\longrightarrow(\mathcal P_2(\mathbb R^d),\mathcal W_2),
\qquad
\rev{\Phi_{T,t}}(\mu_0,w):=\mathcal L(X_t^{\mu_0,w}),
\]
is well defined and continuous. The corresponding restricted map for the
equation posed on $[0,t]$ is
\[
\overline\Phi_t:(\mathcal P_m(\mathbb R^d),\mathcal W_m)
\times\bigl(C^{0,\alpha}([0,t];\mathbb R^{q+1}),\rho_\alpha\bigr)
\longrightarrow(\mathcal P_2(\mathbb R^d),\mathcal W_2),
\qquad
\overline\Phi_t(\mu_0,v):=\mathcal L(X_t^{\mu_0,v}),
\]
which is also well defined and continuous. The two maps satisfy the causality
identity
\[
\rev{\Phi_{T,t}}(\mu_0,w)
=\overline\Phi_t(\mu_0,r_t w).
\]

When the deterministic driver is replaced by the random It\^o lift of the
\rev{time augmented} common noise, this pathwise law coincides almost surely with
the conditional law:
\[
\rev{\Phi_{T,t}}\bigl(\mu_0,
\widehat{W}^{0,\mathrm{It\hat{o}}}(\omega^0)\bigr)
=
\mathcal{L}(X_t\mid\mathcal{F}_t^{W^0})(\omega^0),
\]
where $X$ solves the associated classical It\^o common noise
McKean--Vlasov equation. No drift correction appears in this identification.
\end{proposition}

\begin{remark}[Moment exponent]
If the cited theory applies with $m=2$, then the source measure space is
$(\mathcal P_2,\mathcal W_2)$. If it requires $m>2$, the natural source is
$(\mathcal P_m,\mathcal W_m)$. Since
$\mathcal W_2(\mu,\nu)\leq\mathcal W_m(\mu,\nu)$ for $m\geq2$, continuity of
the resulting law map still holds with values in
$(\mathcal P_2,\mathcal W_2)$.
\end{remark}

\begin{proof}
Fix an initial law $\mu_0$. Choose a probability space carrying an
idiosyncratic Brownian motion $(B_s)_{s\le T}$ and an independent random
variable $\xi$ with $\mathcal L(\xi)=\mu_0$ and finite $m$th moment. For each
deterministic $w\in C^{0,\alpha}([0,T];\mathbb R^{q+1})$, the
\rev{well posedness} theorem \cite[Theorem~3.10]{friz2025mckean} gives a unique
solution $X^{\mu_0,w}$ in the corresponding $L^m$ controlled
rough path class, with
$\sup_{s\le T}\|X_s^{\mu_0,w}\|_{L^m}<\infty$. The driver need not
be geometric, which permits the It\^o lift.

Since $m\ge2$, the law of $X_t^{\mu_0,w}$ belongs to
$\mathcal P_m(\mathbb R^d)\subset\mathcal P_2(\mathbb R^d)$, and hence
\[
\rev{\Phi_{T,t}}(\mu_0,w):=\mathcal L(X_t^{\mu_0,w})
\]
is well defined.

The restriction of $X^{\mu_0,w}$ to $[0,t]$ solves the same frozen
equation driven by $r_t w$. Uniqueness on $[0,t]$ therefore gives
\[
\rev{\Phi_{T,t}}(\mu_0,w)
=\overline\Phi_t(\mu_0,r_t w).
\]
The restricted map is continuous by the same stability argument used below.

Let $(\mu_0^n,w^n)\to(\mu_0,w)$ in
$\mathcal W_m\times\rho_\alpha$. By the gluing lemma, the initial variables
may be realized on a common probability space, together with an independent
idiosyncratic Brownian motion, so that
\[
\mathcal L(\xi^n)=\mu_0^n,
\qquad
\mathcal L(\xi)=\mu_0,
\qquad
\|\xi^n-\xi\|_{L^m}=\mathcal W_m(\mu_0^n,\mu_0).
\]
The convergent rough path sequence is bounded in rough path norm. The
continuous dependence theorem \cite[Theorem~3.11]{friz2025mckean}, applied
with the fixed coefficients and control, therefore yields
\[
\sup_{s\le T}\|X_s^{\mu_0^n,w^n}
-X_s^{\mu_0,w}\|_{L^m}\longrightarrow0.
\]
Since $m\ge2$, convergence also holds in $L^2$. At time $t$, the chosen
coupling consequently gives
\[
\mathcal W_2\!\left(
\rev{\Phi_{T,t}}(\mu_0^n,w^n),\rev{\Phi_{T,t}}(\mu_0,w)
\right)
\le
\|X_t^{\mu_0^n,w^n}-X_t^{\mu_0,w}\|_{L^2}
\longrightarrow0.
\]
This proves the asserted continuity into
$(\mathcal P_2(\mathbb R^d),\mathcal W_2)$.

Uniqueness for the frozen equation shows that
$\rev{\Phi_{T,t}}(\mu_0,w)$ depends only on the initial law and the
deterministic rough driver, not on their chosen realization.

The Brownian randomization argument of Friz--Hocquet--L\^e
\cite[Section~3.3]{friz2025mckean}, together with the progressively measurable
version of the frozen solutions supplied by
\cite[Proposition~3.16]{friz2025mckean} and the identification theorem
\cite[Theorem~3.5]{friz2025randomisation}, permits evaluation of the frozen
solution at the random It\^o Brownian rough path. The resulting process solves
the classical common noise McKean--Vlasov equation, and
\[
\rev{\Phi_{T,t}}\bigl(\mu_0,
\widehat{W}^{0,\mathrm{It\hat{o}}}(\omega^0)\bigr)
=\mathcal{L}(X_t\mid\mathcal{F}_t^{W^0})(\omega^0),
\qquad \mathbb{P}^0\text{ a.s.}
\]

The identification is formulated for the It\^o lift and coefficients with no
direct dependence on the driver, as assumed here. Consequently, it recovers
the classical It\^o common noise system without a drift correction. This
proves existence, causality, continuity, and the claimed conditional law
identification.
\end{proof}

The proposition keeps the admissible control $a$ fixed. Allowing the control
to vary would require a topology and a regularity class for the controls. If
the control enters the rough coefficient, mere measurability in time is not
enough in general because the rough integrand must satisfy the controlled
rough path regularity used in the stability theorem. For an admissible class
on which that estimate applies, the same argument also gives continuity in
the control variable; see \cite{friz_le_zhang_2024} for the pathwise control
setting.

Under the stronger assumptions of Bugini--Friz--Stannat
\cite{bugini2025nonlinear}, the same law flow is also characterized as the
unique solution of the associated nonlinear rough Fokker--Planck equation.

\begin{remark}[Itô and Stratonovich lifts]
The identification result of \cite[Section~3.3]{friz2025mckean} is formulated
for the It\^o lift of the \rev{time augmented} common noise. Brownian randomization
then recovers the classical It\^o common noise equation without a drift
correction. At the second level,
\[
\widehat{\mathbb W}^{0}_{s,t}
=
\widehat{\mathbb W}^{0,\mathrm{It\hat{o}}}_{s,t}
+\frac12 (t-s)
\begin{pmatrix}
0 & 0 \\
0 & I_q
\end{pmatrix}
\]
and the first levels coincide. The translation is continuous and invertible
on its image, so using the It\^o lift for analytic
identification and the Stratonovich lift for signature features loses no
information. It acts on the driver only; no \rev{coefficient level}
It\^o--Stratonovich correction is applied to $b$, $\sigma$, or $f$ in the
approximation theorem. A coefficient conversion for measure dependent common
noise would additionally involve spatial and Lions derivative terms and lies
outside the present scope. Section~\ref{sec:ito-strat-num} checks the state
dependent case numerically.
\end{remark}

\subsection{Conditional cylindrical networks}

The approximation architecture has two components. The first component learns
the conditional law from finite information about the initial distribution
and the \rev{common noise path}. The second component evaluates the target
functional from the current state and finitely many integrals against the
predicted law. This separation is useful because the same learned law can be
reused for different states $x$ and, after retraining only the second
component, for different functionals $V$.

Let $(\psi_i)_{i\geq1}\subset C_b(\mathbb R^d)$ be a fixed countable family
that separates probability measures. For $L\in\mathbb N$, the initial law is
encoded by
\[
E_L(\mu_0)
=\bigl(\langle\psi_1,\mu_0\rangle,\ldots,
\langle\psi_L,\mu_0\rangle\bigr)\in\mathbb R^L,
\qquad
\langle\psi_i,\mu_0\rangle
=\int_{\mathbb R^d}\psi_i(y)\,\mu_0(dy).
\]
The proof of Theorem~1 uses the sine and cosine functions associated with rational
frequencies, which form one explicit separating family. The information from
the common noise is represented by the truncated signature $S_M(r_t w)$
defined in Section~2.2. Combining the two inputs gives the \rev{finite dimensional}
feature map
\[
g_M(\mu_0,r_t w)
:=\bigl(E_L(\mu_0),S_M(r_t w)\bigr)\in\mathbb R^N,
\qquad
N=L+d_{\mathrm{sig}}(M,q).
\]

The conditional network
\[
\widehat N:\mathbb R^N\longrightarrow\mathcal P_2(\mathbb R^d)
\]
maps these features to a probability measure intended to approximate
$\mu_t(w)=\overline\Phi_t(\mu_0,r_t w)$. In the construction below it has a
finite Gaussian mixture output of the form
\[
\widehat N(z)
=\sum_{j=1}^{J}\pi_j(z)\,
\mathcal N\bigl(m_j(z),\Sigma_j(z)\bigr),
\qquad z\in\mathbb R^N,
\]
where $\pi_j(z)\geq0$, $\sum_{j=1}^{J}\pi_j(z)=1$, and each
$\Sigma_j(z)$ is positive definite. Thus the network does not return a single
point estimate of the state. It returns an entire probability law, allowing
multimodality, conditional variances, and distributional asymmetry to be
retained before the functional is evaluated.

The proof of Lemma~1 uses only the subclass in which the Gaussian means and
covariances are fixed and the mixture weights depend on $z$ through a neural
network. This subclass is sufficient for the density argument. The numerical
implementation uses the richer standard parametrization above, in which the
weights, means, and covariance matrices all depend on the input features.

The predicted law is then processed by a cylindrical network in the sense of
Pham and Warin \cite{pham2022meanfield}. For a latent dimension
$k\in\mathbb N$, let
$\varphi=(\varphi_1,\ldots,\varphi_k):\mathbb R^d\to\mathbb R^k$ be bounded
and Lipschitz, and define componentwise
\[
\langle\varphi,\mu\rangle
:=\int_{\mathbb R^d}\varphi(y)\,\mu(dy)
=\bigl(\langle\varphi_1,\mu\rangle,\ldots,
\langle\varphi_k,\mu\rangle\bigr)\in\mathbb R^k.
\]
Boundedness makes these integrals well defined for every probability measure,
while the Lipschitz property makes the map continuous in the Wasserstein
topology. Indeed, for each component,
\[
\bigl|\langle\varphi_i,\mu\rangle
-\langle\varphi_i,\nu\rangle\bigr|
\leq \operatorname{Lip}(\varphi_i)\,\mathcal W_1(\mu,\nu)
\leq \operatorname{Lip}(\varphi_i)\,\mathcal W_2(\mu,\nu).
\]

Finally, an outer feedforward network
$\Psi:\mathbb R^d\times\mathbb R^k\to\mathbb R^p$ combines the state $x$
with these law statistics. The resulting approximating class is
\[
\widehat{V}(x,\mu_0,w)
=\Psi\Bigl(x,\langle\varphi,
\widehat{N}(g_M(\mu_0,r_t w))\rangle\Bigr).
\]
Writing
$\widehat\mu_t(\mu_0,w):=\widehat N(g_M(\mu_0,r_t w))$, its information
flow can be summarized as
\[
(\mu_0,r_t w)
\xrightarrow{\;g_M\;}
\mathbb R^N
\xrightarrow{\;\widehat N\;}
\widehat\mu_t
\xrightarrow{\;\langle\varphi,\cdot\rangle\;}
\mathbb R^k
\xrightarrow{\;\Psi(x,\cdot)\;}
\widehat V(x,\mu_0,w).
\]
The feature orders $L$ and $M$, the number of mixture components $J$, the
latent dimension $k$, and the network sizes are allowed to depend on the
desired accuracy. The theorem below is qualitative and does not assert a
rate for any of these quantities.

Signature encodings of the common noise have recently been used in
numerical schemes for conditional McKean--Vlasov dynamics: Agram and Rems
\cite{agram2025conditional} combine neural networks with path signatures for the
estimation of conditional expectations in the control of conditional
McKean--Vlasov jump diffusions, while Hu, Jin, Laurière and Zhang
\cite{hu2025deepsignature} represent the conditional law through the truncated
log signature of the \rev{time augmented} common noise within a fictitious play
scheme for McKean--Vlasov FBSDEs. The present work is complementary: it
provides a universal approximation guarantee for a single conditional network
class of the form \(\widehat V\) above, the continuity input being supplied by
the rough path stability theory of \cite{friz2025mckean}.

\section{Conditional cylindrical universal approximation}

This section establishes the main universal approximation result. The theorem
shows that continuous \rev{square integrable} functionals of the state and the
conditional law can be approximated in the relevant $L^2$ space by conditional
cylindrical networks. Its proof combines compact truncation, finite Fourier and
signature representations, Gaussian mixture approximation of the conditional
law, and cylindrical approximation of the target functional. The argument is
qualitative and does not provide approximation rates.

\rev{We first establish the two compact approximation lemmas used in the proof.}

\medskip
\noindent\emph{Finite feature factorization.}
Recall from Section~2.5 that
\[
g_M(\mu_0,\mathbf v)
=\bigl(E_L(\mu_0),S_M(\mathbf v)\bigr)\in\mathbb R^N,
\qquad N=L+d_{\mathrm{sig}}(M,q).
\]
For completeness, the separating family used in $E_L$ can be chosen by
enumerating the rational frequencies
$\{k_n\}_{n\geq1}\subset\mathbb Q^d$ and setting
\[
\psi_{2n-1}(x)=\cos\langle k_n,x\rangle,
\qquad
\psi_{2n}(x)=\sin\langle k_n,x\rangle.
\]
These functions are bounded and Lipschitz. Equality of their integrals for
all $i$ gives equality of the characteristic functions on $\mathbb Q^d$ and
hence, by continuity and the uniqueness theorem for characteristic functions
\cite[Chapter~5, Section~26]{billingsley2008probability}, equality of the
measures.

\begin{lemma}[Finite factorization and MDN realization]
Let
\(
K_\mu\subset\mathcal P_m(\mathbb R^d),
\qquad
K_\omega^t
\subset
C^{0,\alpha}_g([0,t];\mathbb R^{q+1})
\)
be compact, and assume that every element of $K_\omega^t$ is
\rev{time augmented}, i.e. its first coordinate is $s\mapsto s$. Set
\(
\mathcal X_t:=K_\mu\times K_\omega^t,
\qquad
d_{\mathcal X_t}
:=
\mathcal W_m+\rho_\alpha.
\)
For every $\theta>0$, there exist $L,M\in\mathbb N$ and a conditional
network
\(
\widehat{N}:\mathbb R^N\to\mathcal P_2(\mathbb R^d)
\)
with finite Gaussian mixture output such that
\[
\sup_{(\mu_0,\mathbf v)\in\mathcal X_t}
\mathcal W_2\left(
\overline{\Phi}_t(\mu_0,\mathbf v),
\widehat{N}(g_M(\mu_0,\mathbf v))
\right)
\leq\theta.
\]
Moreover, the image $\widehat N(g_M(\mathcal X_t))$ is  compact in
$(\mathcal P_2(\mathbb R^d),\mathcal W_2)$, with second moments uniformly
bounded.
\end{lemma}

\begin{proof}
Write
\[
g_\infty=(E_\infty,S_\infty):
\mathcal X_t\longrightarrow\mathbb R^{\mathbb N}
\]
for the concatenation of all Fourier moments of $\mu_0$ and all signature
coordinates of $\mathbf v$. Choose the enumeration compatibly with the
finite blocks, so that for every pair $(L,M)$ an initial block consists of
\[
(E_L(\mu_0),S_M(\mathbf{v})) = g_M(\mu_0,\mathbf{v})
\]
with $N=L+d_{\mathrm{sig}}(M,q)$. Equip
\(\mathbb R^{\mathbb N}\) with the product metric
\[
\rho(a,b)
:=
\sum_{j\geq 1}2^{-j}\bigl(|a_j-b_j|\wedge 1\bigr).
\]

\smallskip

\noindent\emph{The full feature map.}
Each map
\[
\mu_0 \longmapsto \langle \psi_i,\mu_0\rangle
\]
is $\mathcal W_m$ continuous, since $\psi_i$ is bounded and Lipschitz and
$\mathcal W_1\leq\mathcal W_m$;
each signature coordinate is \(\rho_\alpha\) continuous, by continuity of the
signature in the rough path topology. Hence \(g_\infty\) is continuous and,
since \(\mathcal X_t\) is compact,
\[
G_\infty := g_\infty(\mathcal X_t)
\]
is compact in \((\mathbb R^{\mathbb N},\rho)\).

The map $g_\infty$ is injective: $E_\infty$ is injective because the
Fourier family separates probability measures, while $S_\infty$ is
injective on the class of \rev{time augmented} geometric rough paths over
$[0,t]$. A continuous bijection from a compact space to a Hausdorff space is
a homeomorphism onto its image
\cite[Section~26, Theorem~26.6]{munkres2000topology}; hence
\[
g_\infty^{-1}:G_\infty\longrightarrow \mathcal X_t
\]
is continuous. Since \(G_\infty\) is compact, \(g_\infty^{-1}\) is uniformly
continuous. We denote by \(\omega_*\) a modulus of continuity of
\(g_\infty^{-1}\).

\smallskip

\noindent\emph{Finite truncation.}
If $x,x'\in\mathcal X_t$ satisfy $\delta\leq1$ and
\[
|g_M(x)-g_M(x')|_\infty \leq \delta,
\]
then
\[
\rho(g_\infty x,g_\infty x')
\leq
\delta \sum_{j\leq N}2^{-j}
+
\sum_{j>N}2^{-j}
\leq
\delta + 2^{-N}.
\]
Therefore, by uniform continuity of $g_\infty^{-1}$,
\begin{equation}
d_{\mathcal X_t}(x,x')
\leq
\omega_*(\delta+2^{-N}).
\label{eq:factorization}
\end{equation}
Thus $\overline{\Phi}_t$ factors approximately through \(g_M\): if the finite variables
are close, then their preimages are close in \(\mathcal X_t\), and hence the laws
\(\overline{\Phi}_t(\cdot)\) are close by continuity of \(\overline{\Phi}_t\).

\smallskip

\noindent\emph{A finite family of nodal mixtures.}
By \eqref{eq:factorization}, we can fix $N$, and hence $L,M$, and
a radius \(r>0\) such that
\begin{equation}
|g_M(x)-g_M(x')|_\infty \leq r
\quad\Longrightarrow\quad
\mathcal W_2\bigl(\overline\Phi_t(x),\overline\Phi_t(x')\bigr)
\leq \frac{\theta}{4}.
\label{eq:local-law-error}
\end{equation}
Indeed, $\overline\Phi_t$ is uniformly continuous on the compact metric space
$(\mathcal X_t,d_{\mathcal X_t})$. Let $\omega_\Phi$ denote a modulus of
continuity. Combining it with \eqref{eq:factorization}, if
\(
\bigl\|g_M(x)-g_M(x')\bigr\|_\infty \leq \delta,
\)
then
\( 
\mathcal W_2\bigl(\overline\Phi_t(x),\overline\Phi_t(x')\bigr)
\leq
\omega_\Phi\!\left(\omega_*\bigl(\delta+2^{-N}\bigr)\right).
\)
It therefore suffices to pick \(N\) of the admissible form so large that
\( 
\omega_\Phi\!\left(\omega_*(2^{-N+1})\right)
\leq \rev{\frac{\theta}{4}},
\)
and to set \(r:=2^{-N}\).

The compact set
\[
G_N := g_M(\mathcal X_t)\subset \mathbb R^N
\]
is covered by finitely many balls
\[
B(z_1,r),\ldots,B(z_J,r),
\qquad z_j\in G_N.
\]
Choose $x_j\in\mathcal X_t$ such that $g_M(x_j)=z_j$. Finitely supported
probability measures are dense in
$(\mathcal P_2(\mathbb R^d),\mathcal W_2)$
\cite[Theorem~6.18]{villani_ot}. Moreover, if
\[
\nu=\sum_{\ell=1}^{K}p_\ell\delta_{y_\ell},
\qquad
\nu_\varepsilon
=
\sum_{\ell=1}^{K}p_\ell
\mathcal N(y_\ell,\varepsilon^2I_d),
\]
then coupling each atom $y_\ell$ with
$y_\ell+\varepsilon Z$, where $Z\sim\mathcal N(0,I_d)$, gives
\begin{equation}
\mathcal W_2(\nu,\nu_\varepsilon)
\leq
\left(\mathbb E|\varepsilon Z|^2\right)^{1/2}
=\sqrt d\,\varepsilon.
\label{eq:gaussian-smoothing}
\end{equation}
Approximating $\overline\Phi_t(x_j)$ first by such a measure $\nu$ and then
using \eqref{eq:gaussian-smoothing} with sufficiently small $\varepsilon$
shows that finite Gaussian mixtures are dense in $\mathcal P_2(\mathbb R^d)$.
Thus we may choose mixtures $m_j$ satisfying
\[
\mathcal W_2\bigl(m_j,\overline\Phi_t(x_j)\bigr)
\leq \frac{\theta}{4}.
\]

Since
\[
\overline\Phi_t(x_j)\in\overline\Phi_t(\mathcal X_t)=:K
\]
and \(K\) is compact, writing
\[
C_K := \sup_{\mu\in K}\int_{\mathbb R^d}|x|^2\,\mu(dx)<\infty,
\]
we have
\[
\left(\int_{\mathbb R^d}|x|^2\,m_j(dx)\right)^{1/2}
=
\mathcal W_2(m_j,\delta_0)
\leq
\mathcal W_2\bigl(m_j,\overline\Phi_t(x_j)\bigr)
+
\mathcal W_2\bigl(\overline\Phi_t(x_j),\delta_0\bigr)
\leq
\rev{\frac{\theta}{4}}+\sqrt{C_K}.
\]
Hence
\[
\sup_{1\leq j\leq J}
\int_{\mathbb R^d}|x|^2\,m_j(dx)
\leq
\left(\rev{\frac{\theta}{4}}+\sqrt{C_K}\right)^2.
\]
Thus the measures \(m_j\) have uniformly bounded second moments.

\smallskip

\noindent\emph{Construction of a continuous factor.}
For $z$ in a neighbourhood of $G_N$, define
\[
\eta_j(z):=\bigl(r-|z-z_j|\bigr)_+,
\qquad
\chi_j(z):=\frac{\eta_j(z)}{\sum_{k=1}^J\eta_k(z)}.
\]
The denominator is positive on an open neighbourhood of $G_N$. Hence the
$\chi_j$ are continuous, sum to one, and satisfy
$\chi_j(z)>0$ only when $|z-z_j|<r$. Define
\[
F(z)
:=
\sum_{j=1}^J \chi_j(z)m_j.
\]
This is a convex combination of finitely many Gaussian mixtures, hence is
again a finite Gaussian mixture. Its image has uniformly bounded second
moments.

\smallskip

\noindent
\emph{Continuity.}
The square Wasserstein distance is convex in the mixture sense: for weights
\(\lambda,\lambda'\) on the same finite set of components,
\begin{equation}
\mathcal W_2^2
\left(
\sum_{j=1}^J \lambda_j m_j,
\sum_{j=1}^J \lambda'_j m_j
\right)
\leq
D^2 \frac{1}{2}\|\lambda-\lambda'\|_1,
\qquad
D:=\max_{j,k}\mathcal W_2(m_j,m_k)<\infty.
\label{eq:mixture-continuity}
\end{equation}
Since the functions $\chi_j$ are continuous, this estimate shows that $F$ is
continuous.

\smallskip

\noindent
\emph{Accuracy.}
Let \(z=g_M(x)\), $x \in \mathcal{X}_t$. Only the indices \(j\) such that \(\chi_j(z)>0\) contribute.
For such \(j\), since the partition is subordinate to the cover, we have
\[
|z-z_j|<r.
\]
Hence, by \eqref{eq:local-law-error},
\[
\mathcal W_2\bigl(\overline\Phi_t(x),\overline\Phi_t(x_j)\bigr)
\leq \rev{\frac{\theta}{4}}.
\]
Therefore, for all such \(j\),
\[
\mathcal W_2\bigl(m_j,\overline\Phi_t(x)\bigr)
\leq
\mathcal W_2\bigl(m_j,\overline\Phi_t(x_j)\bigr)
+
\mathcal W_2\bigl(\overline\Phi_t(x_j),\overline\Phi_t(x)\bigr)
\leq
\rev{\frac{\theta}{4}}
+
\rev{\frac{\theta}{4}}
=
\rev{\frac{\theta}{2}}.
\]
Mixing optimal couplings gives
\begin{equation}
\mathcal W_2\bigl(F(g_M(x)),\overline\Phi_t(x)\bigr)
\leq
\max_{j:\chi_j(z)>0}
\mathcal W_2\bigl(m_j,\overline\Phi_t(x)\bigr)
\leq
\frac{\theta}{2},
\label{eq:law-error}
\end{equation}
uniformly on \(\mathcal X_t\).

\smallskip

\noindent\emph{Realization by a conditional network.}
The means and covariances of the mixtures $m_j$ are fixed, while the weight
map $\chi=(\chi_1,\ldots,\chi_J)$ is continuous on $G_N$. For a small
$a>0$, set
\[
\chi_j^a(z):=\frac{\chi_j(z)+a}{1+Ja}.
\]
The map $z\mapsto(\log\chi_j^a(z))_{j=1}^J$ is continuous. The classical
universal approximation theorem \cite{hornik1991approximation}, followed by
a softmax output layer, therefore gives network weights $\pi_j(z)$ satisfying
\[
\sup_{z\in G_N}\|\pi(z)-\chi(z)\|_1\leq\varepsilon_\lambda,
\]
after first choosing $a$ and then the network approximation sufficiently
small. Estimate \eqref{eq:mixture-continuity} transports this into a
$\mathcal W_2$ error bounded by
\[
\frac{D}{\sqrt{2}}\sqrt{\varepsilon_\lambda}\leq\rev{\frac{\theta}{2}}.
\]
We obtain a conditional network with Gaussian mixture output
\[
\widehat N:\mathbb R^N \longrightarrow \mathcal P_2(\mathbb R^d)
\]
such that, by \eqref{eq:law-error} and the triangle inequality,
\[
\sup_{(\mu_0,\mathbf v)\in \mathcal X_t}
\mathcal W_2
\left(
\overline\Phi_t(\mu_0,\mathbf v),
\widehat N(g_M(\mu_0,\mathbf v))
\right)
\leq \theta.
\]

Finally, \(\widehat N\) is continuous, since it is a network composed with a
continuous parametrization of the mixture. Therefore
\[
\widehat N(g_M(\mathcal X_t))
\]
is the continuous image of the compact set \(g_M(\mathcal X_t)\), hence is compact in
$(\mathcal P_2(\mathbb R^d),\mathcal W_2)$. Its second moments are uniformly
bounded by the bounds for the finite family $(m_j)_{j=1}^J$.
\end{proof}

\medskip
\noindent\emph{Uniform cylindrical approximation.}
The next lemma is the compact set version of the cylindrical approximation
principle of Pham and Warin \cite{pham2022meanfield}.

\begin{lemma}[Uniform cylindrical approximation]
Let \(K_x\subset \mathbb R^d\) and
\(K_\mu\subset \mathcal P_2(\mathbb R^d)\) be compact, and let
\[
G:\mathbb R^d\times \mathcal P_2(\mathbb R^d)\longrightarrow \mathbb R^p
\]
be continuous. For every \(\beta>0\), there exist \(k\in\mathbb N\), inner
functions
\[
\varphi=(\varphi_1,\ldots,\varphi_k),
\qquad
\varphi_i\in C_b(\mathbb R^d),
\]
bounded and Lipschitz, and an outer network
\[
\Psi:\mathbb R^d\times \mathbb R^k\longrightarrow \mathbb R^p
\]
such that
\[
\sup_{(x,\mu)\in K_x\times K_\mu}
\left|
G(x,\mu)-\Psi\bigl(x,\langle \varphi,\mu\rangle\bigr)
\right|
\leq \beta .
\]
\end{lemma}

\begin{proof}
First consider cylindrical functions generated by the Fourier family. Let
$\mathcal D=\{\psi_i\}_{i\geq1}$ be the Fourier family, which is
bounded, Lipschitz, and separates points of
\(\mathcal P_2(\mathbb R^d)\). For \(\vartheta\in\mathcal D\), the map
\[
\mu\longmapsto \langle \vartheta,\mu\rangle
\]
is \(\mathcal W_2\) continuous. Indeed, by the Kantorovich--Rubinstein
duality \cite[Theorem~1.14]{villani_topics_ot},
\[
\left|
\langle \vartheta,\mu\rangle-\langle \vartheta,\nu\rangle
\right|
\leq
\operatorname{Lip}(\vartheta)\mathcal W_1(\mu,\nu)
\leq
\operatorname{Lip}(\vartheta)\mathcal W_2(\mu,\nu).
\]

Consider the set \(\mathcal A\) of functions of the form
\[
(x,\mu)
\longmapsto
P\bigl(x,\langle \vartheta_1,\mu\rangle,\ldots,
          \langle \vartheta_n,\mu\rangle\bigr),
\]
where
\[
n\in\mathbb N,\qquad
\vartheta_i\in\mathcal D,
\qquad
P \text{ is a polynomial}.
\]
The set \(\mathcal A\) is a subalgebra of
\(C(K_x\times K_\mu)\) containing the constants. It separates points:
if
\[
(x,\mu)\neq (x',\mu'),
\]
then either \(x\neq x'\), in which case the points are separated by one
coordinate function \(x_i\), or \(\mu\neq \mu'\), in which case there exists
\(\vartheta_i\in\mathcal D\) such that
\[
\langle \vartheta_i,\mu\rangle
\neq
\langle \vartheta_i,\mu'\rangle .
\]
Therefore, by the Stone--Weierstrass theorem
\cite[Chapter~7, Theorem~7.32]{rudin_pma}, $\mathcal A$ is dense in
$C(K_x\times K_\mu)$ for the uniform norm. Applying this componentwise to
$G$ and taking the union of the finitely many test functions involved, we may
choose a vector valued polynomial $P$ and
$\vartheta_1,\ldots,\vartheta_n$ such that
\[
\sup_{(x,\mu)\in K_x\times K_\mu}
\left|
G(x,\mu)
-
P\bigl(x,\langle \varphi_0,\mu\rangle\bigr)
\right|
\leq
\frac{\beta}{2},
\]
where
\[
\varphi_0:=(\vartheta_1,\ldots,\vartheta_n)
\in C_b(\mathbb R^d;\mathbb R^n)
\]
is bounded and Lipschitz.

It remains to replace the polynomial by a network. The set
\[
K_x\times \langle \varphi_0,K_\mu\rangle
\subset
\mathbb R^d\times \mathbb R^n
\]
is compact, since it is the continuous image of a compact set. Moreover,
\[
|\langle \varphi_0,\mu\rangle|
\leq
\|\varphi_0\|_\infty .
\]
The polynomial \(P\) is continuous on this compact set. Therefore, by the
classical \rev{finite dimensional} universal approximation theorem
\cite{hornik1991approximation}, \rev{using a Lipschitz nonconstant activation},
there exists a
neural network
\[
\Psi:\mathbb R^d\times \mathbb R^n\longrightarrow \mathbb R^p
\]
such that
\[
\sup_{(x,z)\in K_x\times \langle \varphi_0,K_\mu\rangle}
|P(x,z)-\Psi(x,z)|
\leq
\frac{\beta}{2}.
\]
Setting
\[
\varphi:=\varphi_0,
\]
we obtain, for every \((x,\mu)\in K_x\times K_\mu\),
\[
\left|
G(x,\mu)
-
\Psi\bigl(x,\langle \varphi,\mu\rangle\bigr)
\right|
\leq
\beta.
\]
This proves the lemma.
\end{proof}

\medskip

\begin{theorem}[Conditional cylindrical universal approximation]
\label{thm:main}
Fix $t\in(0,T]$ and let $m\geq2$ be as in Proposition~1. Let $\nu$ be a
probability measure on $\mathcal P_m(\mathbb R^d)$, and let
\[
V:\mathbb R^d\times\mathcal P_2(\mathbb R^d)\longrightarrow\mathbb R^p
\]
be continuous and satisfy the mean square integrability condition of
Section~2.3. For
\[
\mu_t(\mathbf w):=\overline\Phi_t(\mu_0,r_t\mathbf w),
\]
where $\overline\Phi_t$ is the restricted conditional law map of
Proposition~1, the following holds. For every $\varepsilon>0$, there exist
$L,M\in\mathbb N^*$, a conditional network $\widehat N$, a latent dimension
$k$, bounded Lipschitz inner functions $\varphi$, and an outer network $\Psi$
such that
\[
\int_{\mathcal P_m(\mathbb R^d)}
\int_{C_g^{0,\alpha}([0,T];\mathbb R^{q+1})}
\int_{\mathbb{R}^d}
\Bigl|V(x,\mu_t(\mathbf w))
-\Psi\bigl(x,\langle\varphi,
\widehat N(g_M(\mu_0,r_t\mathbf w))\rangle\bigr)\Bigr|^2
\,\mu_t(\mathbf w)(dx)\,\mathbf P^0(d\mathbf w)\,\nu(d\mu_0)
\leq\varepsilon,
\]
where
\[
g_M(\mu_0,r_t\mathbf w)
=\bigl(E_L(\mu_0),S_M(r_t\mathbf w)\bigr)\in\mathbb R^N,
\qquad N=L+d_{\mathrm{sig}}(M,q).
\]
\end{theorem}

\begin{proof}
The argument first truncates $V$ and restricts the data to a compact set.
It then factors the conditional law map through finitely many features,
approximates the functional by a cylindrical network on the resulting compact
set, and finally combines the estimates in $L^2$. We abbreviate
$\mu_t(\mathbf w):=\overline\Phi_t(\mu_0,r_t\mathbf w)$.

\medskip
\noindent\emph{Truncation of $V$.}

For \(R>0\), define the truncation
\[
V_R(x,\mu)=
\begin{cases}
V(x,\mu), & \text{if } |V(x,\mu)| \le R,\\[8pt]
R\,\dfrac{V(x,\mu)}{|V(x,\mu)|}, & \text{if } |V(x,\mu)| > R.
\end{cases}
\]
Equivalently, $V_R=\pi_R\circ V$, where $\pi_R$ is the radial projection
onto the closed ball of radius $R$ in $\mathbb R^p$.

Then $|V_R|\leq R$ everywhere, $|V_R|\leq|V|$, and $V_R\to V$
pointwise as $R\to\infty$.

Set
\[
f_R(x,\mu) := |V(x,\mu) - V_R(x,\mu)|^2.
\]
Then \(f_R \to 0\) pointwise, with the domination
\[
f_R(x,\mu)=|V - V_R|^2 \le |V|^2
\quad\text{and}\quad
|V|^2 \in L^1.
\]

By dominated convergence,
\[
\|V - V_R\|_{L^2}^2
=
\iiint |V - V_R|^2\, \mu_t(\mathbf w)(dx)\, d(\nu \otimes \mathbf P^0)
\longrightarrow 0
\qquad \text{as } R \to \infty.
\]

Fix \(R\) large enough that
\begin{equation}
\|V - V_R\|_{L^2}^2 \le \frac{\varepsilon}{8}.
\label{eq:truncation}
\end{equation}
Set
\[
C_p:=(1+\sqrt p)^2.
\]

\medskip
\noindent\emph{Reduction to a compact set.}
The spaces $(\mathcal P_m(\mathbb R^d),\mathcal W_m)$ and
$(C_g^{0,\alpha}([0,T];\mathbb R^{q+1}),\rho_\alpha)$ are Polish, and so is
their product. Every probability measure on a Polish space is tight
\cite[Theorem~1.3]{billingsley1999convergence}. Consequently, for every
$\delta>0$ there exist compact sets
\[
K_\mu\subset \mathcal P_m(\mathbb R^d),
\qquad
K_\omega\subset C_g^{0,\alpha}([0,T];\mathbb R^{q+1})
\]
such that
\[
\nu(K_\mu^c)\le \delta,
\qquad
\mathbf P^0(K_\omega^c)\le \delta.
\]

Moreover, $\mathbf P^0$ is carried by the closed subset
\[
\mathcal{T}
:=
\left\{
\mathbf{w} \in C_g^{0,\alpha}
:\ 
\pi_1\!\left(\mathbf{w}_{0,s}^{(1)}\right)=s
\ \text{for all } s \in [0,T]
\right\}
\]
of \rev{time augmented} rough paths. Replacing \(K_\omega\) by the compact set
\(K_\omega \cap \mathcal{T}\) preserves the bound
\(\mathbf{P}^0(K_\omega^c) \leq \delta\); we do so henceforth. On
\(\mathcal{T}\), the full signature is injective~\cite{boedihardjo2014signature}.

Choose
\[
\delta=\frac{\varepsilon}{16C_pR^2}.
\]
By the union bound,
\[
(K_\mu\times K_\omega)^c
=
(K_\mu^c \times C_g^{0,\alpha})\cup(\mathcal P_m\times K_\omega^c),
\]
hence
\begin{equation}
(\nu\otimes \mathbf P^0)\bigl((K_\mu\times K_\omega)^c\bigr)
\leq \nu(K_\mu^c)+\mathbf P^0(K_\omega^c)
\leq 2\delta
=
\frac{\varepsilon}{8C_pR^2}.
\label{eq:compact}
\end{equation}

Set
\(
\mathcal X:=K_\mu\times K_\omega.
\)
This is the high probability compact set for the full data on $[0,T]$.

We now restrict the rough path component to the interval relevant for
the value at time $t$. Define
\[
K_\omega^t:=r_t(K_\omega)
\subset
C^{0,\alpha}_g([0,t];\mathbb{R}^{q+1}),
\qquad
\mathcal X_t:=K_\mu\times K_\omega^t.
\]
Since $r_t$ is continuous and $K_\omega$ compact, $K_\omega^t$ and $\mathcal{X}_t$ are compact. Every element of $K_\omega^t$ is \rev{time augmented}, since the defining property of $\mathcal{T}$ is preserved under restriction to $[0,t]$.

Moreover, by causality of the frozen equation, for every
$(\mu_0,\mathbf{w})\in\mathcal X$,
\[
\mu_t(\mathbf{w})
=
\overline{\Phi}_t(\mu_0,r_t\mathbf{w}),
\qquad
(\mu_0,r_t\mathbf{w})\in\mathcal X_t.
\]
All subsequent factorization arguments will therefore be performed on
$\mathcal X_t$.

\medskip
\noindent\emph{Compactness of the image.}
Since $\overline{\Phi}_t$ is continuous and $\mathcal X_t$ is compact,
\[
K
:=
\overline{\Phi}_t(\mathcal X_t)
\subset \mathcal P_2(\mathbb R^d)
\]
is compact. Moreover,
\(
K
=
\left\{
\mu_t(\mathbf{w}):
(\mu_0,\mathbf{w})\in\mathcal X
\right\}
\)
by causality.

\medskip
\noindent\emph{$L^2$ reassembly.}

We have:
\begin{itemize}
    \item a radius \(R>0\) with \(V_R=\pi_R\circ V\), \(|V_R|\leq R\),
    satisfying \eqref{eq:truncation};
    \item compact sets \(K_\mu,K_\omega\),
    \(\mathcal X=K_\mu\times K_\omega\), and
    \(\mathcal{X}_t=K_\mu\times r_t(K_\omega)\), satisfying
    \eqref{eq:compact}, with
    \[
    K:=\overline{\Phi}_t(\mathcal X_t)
    =\{\mu_t(\mathbf w):(\mu_0,\mathbf w)\in\mathcal X\}
    \]
    compact.
\end{itemize}

\smallskip

\noindent\emph{Spatial truncation.}
Since $K$ is compact in $(\mathcal P_2(\mathbb R^d),\mathcal W_2)$, it is also
relatively compact for the weak topology. Indeed, convergence in $\mathcal W_2$
implies weak convergence of probability measures \cite[Theorem~6.9]{villani_ot}. Hence \(K\) is relatively
compact in \((\mathcal P(\mathbb R^d),\Rightarrow)\), where \(\Rightarrow\)
denotes weak convergence.

By Prokhorov's theorem on the Polish space \(\mathbb R^d\), weak relative
compactness of a family of probability measures is equivalent to uniform
tightness \cite[Chapter~1, Section~5]{billingsley2008probability}.
Therefore \(K\) is uniformly tight. Consequently, for every \(\eta>0\), there
exists a compact set \(C_\eta\subset\mathbb R^d\) such that
\[
\sup_{\mu\in K}\mu(C_\eta^c)\le \eta.
\]
Taking
\[
\eta:=\frac{\varepsilon}{8C_pR^2},
\]
we obtain a compact set \(C_\eta\subset\mathbb R^d\) satisfying
\[
\sup_{\mu\in K}\mu(C_\eta^c)\le \frac{\varepsilon}{8C_pR^2}.
\]
Since \(C_\eta\) is compact in \(\mathbb R^d\), it is bounded. Hence there
exists \(\rho>0\) such that
\[
C_\eta\subset B_\rho:=\{x\in\mathbb R^d:|x|\le \rho\}.
\]
Thus
\[
B_\rho^c\subset C_\eta^c,
\]
and consequently
\[
\sup_{\mu\in K}\mu(B_\rho^c)
\le
\sup_{\mu\in K}\mu(C_\eta^c)
\le
\frac{\varepsilon}{8C_pR^2}.
\]

\smallskip

\noindent\emph{The cylindrical layer.}
The map
\[
V_R=\pi_R\circ V
\]
is continuous on \(\mathbb R^d\times \mathcal P_2(\mathbb R^d)\).
Applying Lemma~2 to $V_R$ on $B_\rho\times K$, with tolerance
\(
\beta:=\frac{\sqrt{\varepsilon}}{8},
\)
there exist \(k\in\mathbb N\), inner functions
\[
\varphi=(\varphi_1,\ldots,\varphi_k)
\in C_b(\mathbb R^d;\mathbb R^k),
\]
bounded and Lipschitz, and an outer network \(\Psi_0\) such that
\[
\sup_{x\in B_\rho,\; m\in K}
\left|
V_R(x,m)
-
\Psi_0\bigl(x,\langle \varphi,m\rangle\bigr)
\right|
\leq
\beta.
\]
Define the coordinatewise clipping map
\[
\bigl(c_R(y)\bigr)_j:=\max\{-R,\min\{y_j,R\}\},
\qquad j=1,\ldots,p,
\]
and set $\Psi:=c_R\circ\Psi_0$. The map $c_R$ is represented exactly by a
coordinatewise ReLU layer and is $1$ Lipschitz. Since $|V_R|\leq R$, it fixes
$V_R$, and therefore
\[
\left|
V_R-\Psi(\cdot,\langle \varphi,\cdot\rangle)
\right|
\leq
\left|
V_R-\Psi_0(\cdot,\langle \varphi,\cdot\rangle)
\right|
\leq
\beta
\]
on $B_\rho\times K$. Moreover, $|\Psi|\leq\sqrt pR$ everywhere. Since the outer network is chosen with a Lipschitz activation and the
clipping map is $1$ Lipschitz, $\Psi$ is globally Lipschitz. We write
\[
\ell:=\operatorname{Lip}(\varphi)<\infty,
\qquad
L_\Psi:=\operatorname{Lip}(\Psi)<\infty.
\]

\smallskip

\noindent\emph{Tolerance for the law approximation.}

Choose \(\theta>0\) such that
\[
L_\Psi\ell\theta
\leq
\tau
:=
\frac{\sqrt{\varepsilon}}{8}.
\]

\smallskip

\noindent\emph{The conditional law network.}
Lemma~1 applied on \(\mathcal{X}_t\) with tolerance \(\theta\) yields orders \(L,M\)
and a conditional network
\[
\widehat N:\mathbb R^N\longrightarrow \mathcal P_2(\mathbb R^d),
\qquad
N=L+d_{\mathrm{sig}}(M,q),
\]
with Gaussian mixture output such that
\[
\sup_{(\mu_0,\mathbf v)\in\mathcal X_t}
\mathcal W_2\!\left(
\overline\Phi_t(\mu_0,\mathbf v),
\widehat N\bigl(g_M(\mu_0,\mathbf v)\bigr)
\right)
\leq\theta.
\]
By causality, for every $(\mu_0,\mathbf w)\in\mathcal X$ the pair
$(\mu_0,r_t\mathbf w)$ belongs to $\mathcal X_t$, so that
\[
\sup_{(\mu_0,\mathbf w)\in\mathcal X}
\mathcal W_2\!\left(
\mu_t(\mathbf w),
\widehat N\bigl(g_M(\mu_0,r_t\mathbf w)\bigr)
\right)
\leq\theta.
\]
We define the approximant in the announced class by
\[
\widehat V(x,\mu_0,\mathbf{w})
:=
\Psi
\left(
x,
\left\langle
\varphi,
\widehat N(g_M(\mu_0,r_t\mathbf{w}))
\right\rangle
\right),
\qquad
g_M(\mu_0,r_t\mathbf{w})=(E_L(\mu_0),S_M(r_t\mathbf{w})).
\]

\smallskip

\noindent\emph{Pointwise estimate on $\mathcal X$.}
Let \((\mu_0,\mathbf{w})\in \mathcal{X}\), and set
\[
\mu_t := \mu_t(\mathbf w)=\overline\Phi_t(\mu_0,r_t\mathbf w),
\qquad
\widehat\nu:=\widehat N(g_M(\mu_0,r_t\mathbf{w})).
\]
Then
\[
\mathcal W_2(\mu_t,\widehat\nu)\leq \theta.
\]
We split
\[
V_R(x,\mu_t)-\widehat V
=
\underbrace{
\left[
V_R(x,\mu_t)
-
\Psi\bigl(x,\langle \varphi,\mu_t\rangle\bigr)
\right]
}_{T_1}
+
\underbrace{
\left[
\Psi\bigl(x,\langle \varphi,\mu_t\rangle\bigr)
-
\Psi\bigl(x,\langle \varphi,\widehat\nu\rangle\bigr)
\right]
}_{T_2}.
\]

For $x\in B_\rho$, we have $|T_1|\leq\beta$ since $\mu_t\in K$. Moreover,
by Kantorovich--Rubinstein and the Lipschitz property of $\varphi$,
    \[
    \left|
    \langle \varphi,\mu_t\rangle
    -
    \langle \varphi,\widehat\nu\rangle
    \right|
    \leq
    \ell\,\mathcal W_1(\mu_t,\widehat\nu)
    \leq
    \ell\,\mathcal W_2(\mu_t,\widehat\nu)
    \leq
    \ell\theta.
    \]

Hence
\[
|T_2|
\leq
L_\Psi
\left|
\langle \varphi,\mu_t\rangle
-
\langle \varphi,\widehat\nu\rangle
\right|
\leq L_\Psi\ell\theta
\leq\tau.
\]

Thus
\[
|V_R(x,\mu_t)-\widehat V|
\leq
\beta+\tau
\qquad
\text{on } B_\rho.
\]

For \(x\in B_\rho^c\), we use the clipping:
\[
|V_R(x,\mu_t)-\widehat V|
\leq
|V_R|+|\widehat V|
\leq R+\sqrt pR.
\]
Integrating with respect to \(\mu_t\), and since \(\mu_t\in K\) gives
\[
\mu_t(B_\rho^c)
\leq
\frac{\varepsilon}{8C_pR^2},
\]
we obtain
\[
\int_{\mathbb R^d}
|V_R(x,\mu_t)-\widehat V|^2\,\mu_t(dx)
\leq
(\beta+\tau)^2+C_pR^2\frac{\varepsilon}{8C_pR^2}
=\frac{\varepsilon}{16}+\frac{\varepsilon}{8}
=\frac{3\varepsilon}{16}.
\]

where we used
\[
(\beta+\tau)^2
=
\left(\frac{\sqrt{\varepsilon}}{4}\right)^2
=
\frac{\varepsilon}{16}.
\]
This bound is uniform on $\mathcal X$.

\smallskip

\noindent\emph{Integration.}
On \(\mathcal{X}\), we have
\[
\int_\mathcal{X}
\left(
\int_{\mathbb R^d}
|V_R-\widehat V|^2\,\mu_t(dx)
\right)
d( \nu\otimes \mathbf P^0)
\leq
\frac{3\varepsilon}{16}(\nu\otimes \mathbf P^0)(\mathcal X)
\leq
\frac{3\varepsilon}{16}.
\]
On $\mathcal X^c$, the global bounds give
\[
|V_R-\widehat V|^2\leq C_pR^2.
\]
Hence
\[
\int_{\mathcal X^c}
\left(
\int_{\mathbb R^d}
|V_R-\widehat V|^2\,\mu_t(dx)
\right)
d(\nu\otimes \mathbf P^0)
\leq
C_pR^2(\nu\otimes \mathbf P^0)(\mathcal X^c)
\leq
C_pR^2\frac{\varepsilon}{8C_pR^2}
=\frac{\varepsilon}{8}.
\]
Therefore
\[
\|V_R-\widehat V\|_{L^2}^2
\leq
\frac{3\varepsilon}{16}+\frac{\varepsilon}{8}
=\frac{5\varepsilon}{16}.
\]
Finally, using
\[
|a-c|^2\leq 2|a-b|^2+2|b-c|^2
\]
with \(b=V_R\), together with \eqref{eq:truncation}, we obtain
\[
\|V-\widehat V\|_{L^2}^2
\leq
2\|V-V_R\|_{L^2}^2
+
2\|V_R-\widehat V\|_{L^2}^2
\leq
2\cdot \frac{\varepsilon}{8}
+
2\cdot \frac{5\varepsilon}{16}
=\frac{7\varepsilon}{8}
<\varepsilon.
\]

This is exactly
\[
\int_{\mathcal P_m(\mathbb R^d)}
\int_{C_g^{0,\alpha}}
\int_{\mathbb R^d}
\left|
V(x,\mu_t(\mathbf{w}))
-
\Psi\left(
x,
\left\langle
\varphi,
\widehat N(g_M(\mu_0,r_t\mathbf w))
\right\rangle
\right)
\right|^2
\mu_t(\mathbf{w})(dx)\,\mathbf P^0(d\mathbf{w})\,\nu(d\mu_0)
\leq
\varepsilon,
\]
where
\[
g_M(\mu_0,r_t\mathbf{w})
=
\bigl(E_L(\mu_0),S_M(r_t\mathbf{w})\bigr).
\]
Thus \(L,M\in\mathbb N^*\), \(\widehat N\), \(k\), \(\varphi\), and \(\Psi\)
have been constructed as required.
\end{proof}
\section{Numerical experiments}
This section evaluates the conditional cylindrical architecture as a
neural density estimation and functional regression pipeline. The
experiments address four questions: whether finite Fourier and signature
features enable accurate prediction of the conditional law; whether the
resulting amortized law representation improves the downstream
approximation of \(V(x,\mu_T)\) relative to an empirical particle plug in;
whether this advantage persists for \rev{non Gaussian} initial laws; and which
feature and architectural choices have the largest effect on performance.
The PyTorch implementation, benchmark configurations, random seeds,
complete per run metrics, and aggregation scripts are available in the
companion repository at \url{https://github.com/HmiouiReda/cylindrical-mckean-vlasov}. 

\subsection{Benchmark models}

The six examples below are designed to test distinct parts of the proposed
architecture. All models have scalar common noise, so $q=1$, and interact
through the conditional mean
$m_s:=\mathbb E[X_s\mid\mathcal F_s^{W^0}]$. The first three examples share
the conditional OU dynamics
\[
 dX_s=a(m_s-X_s)\,ds+\sigma\,dB_s+\sigma_0\,dW_s^0,
 \qquad a=1,\quad \sigma=0.4,\quad \sigma_0=0.3,\quad T=1,
\]
but use different \rev{scenario dependent} initial laws. They provide exact
conditional distributions against which both stages of the architecture can
be evaluated. The remaining examples introduce nonlinear drift,
multiplicative common noise, and a \rev{two dimensional} state.

\begin{example}[Gaussian initial law]
\label{ex:ou-gaussian}
The initial condition satisfies
\[
X_0\sim\mathcal N(m_0,s_0^2),
\qquad
m_0\sim\mathcal U(-1,1),
\qquad
s_0\sim\mathcal U(0.2,0.8).
\]
The conditional terminal law is
\[
\mu_T
=
\mathcal N\!\left(
m_0+\sigma_0W_T^0,\,
s_0^2e^{-2aT}
+\frac{\sigma^2}{2a}\bigl(1-e^{-2aT}\bigr)
\right).
\]
\end{example}

\begin{example}[Uniform initial law]
\label{ex:ou-uniform}
Let
\[
X_0=m_0+U_0,
\qquad
U_0\sim\mathcal U[-h_0,h_0],
\]
where
\[
m_0\sim\mathcal U(-1,1),
\qquad
h_0\sim\mathcal U(0.3,1.0).
\]
Conditionally on \(W^0\), the terminal random variable admits the exact
representation
\[
X_T=m_T+e^{-aT}U_0+Z_T,
\qquad
m_T=m_0+\sigma_0W_T^0,
\]
where
\[
Z_T\sim\mathcal N(0,v_{\mathrm{noise}}),
\qquad
v_{\mathrm{noise}}
=
\frac{\sigma^2}{2a}\bigl(1-e^{-2aT}\bigr).
\]
Hence, the conditional terminal law is an exact uniform--Gaussian
convolution.
\end{example}

\begin{example}[Two component Gaussian mixture initial law]
\label{ex:ou-gmm}
The initial law is
\[
X_0
\sim
p\,\mathcal N(\mu_{1,0},s_c^2)
+
(1-p)\,\mathcal N(\mu_{2,0},s_c^2),
\]
where
\[
m_0\sim\mathcal U(-1,1),
\qquad
p\sim\mathcal U(0.25,0.75),
\qquad
s_c\sim\mathcal U(0.10,0.30),
\qquad
\Delta\sim\mathcal U(0.6,1.8),
\]
and
\[
\mu_{1,0}=m_0-(1-p)\Delta,
\qquad
\mu_{2,0}=m_0+p\Delta.
\]
Thus,
\[
p\mu_{1,0}+(1-p)\mu_{2,0}=m_0
\]
and the component separation is
\(\mu_{2,0}-\mu_{1,0}=\Delta\).
The conditional terminal law remains an exact \rev{two component} Gaussian
mixture with unchanged weights, component means
\[
\mu_{j,T}
=
m_T+e^{-aT}(\mu_{j,0}-m_0),
\qquad j\in\{1,2\},
\]
and common component variance
\[
s_{c,T}^2
=
e^{-2aT}s_c^2+v_{\mathrm{noise}}.
\]
\end{example}

\begin{example}[Double well model]
\label{ex:double-well}
\[
dX_s=\bigl[\kappa(X_s-X_s^3)+\theta(m_s-X_s)\bigr]\,ds+\sigma\,dB_s+\sigma_0\,dW_s^0,
\]
with
\[
\kappa=1,\qquad \theta=0.25,\qquad \sigma=0.5,
\qquad \sigma_0=0.25,\qquad T=2.
\]
The initial law is $X_0\sim\mathcal N(m_0,s_0^2)$, where
$m_0\sim\mathcal U(-0.3,0.3)$ and
$s_0\sim\mathcal U(0.3,0.6)$. The conditional laws are nonlinear and may be
bimodal, since the common noise changes the relative occupancy of the two
wells. We discretize the superlinear drift using the tamed Euler scheme of
Hutzenthaler--Jentzen--Kloeden \cite{hutzenthaler2012tamed}. No \rev{closed form}
conditional law is available, so validation uses the particle reference
described in Subsection~\ref{sec:reference-design}.
\end{example}

\begin{example}[Multiplicative common noise]
\label{ex:multiplicative}
\[
dX_s=a(m_s-X_s)\,ds+\sigma\,dB_s+\sigma_0(X_s)\,dW_s^0
\qquad \text{(It\^o)},\qquad
\sigma_0(x)=\nu_0(1+\beta\sin x),
\]
with $a=1$, $\sigma=0.3$, $\nu_0=0.4$, $\beta=0.8$, and $T=1$. The initial law is \rev{scenario dependent} and Gaussian,
$X_0\sim\mathcal N(m_0,s_0^2)$. The coefficient $\sigma_0$ is smooth,
bounded, and bounded away from zero for $\beta<1$, and it is \rev{state dependent},
so the It\^o--Stratonovich correction
$\frac12\sigma_0\sigma_0'(x)=\frac12\nu_0^2\beta\cos x\,(1+\beta\sin x)$
does not vanish. This is the benchmark that stresses Remark~2; see
Section~\ref{sec:ito-strat-num}.
\end{example}

\begin{example}[Two dimensional conditional OU]
\label{ex:ou-2d}
\[
dX_s=a(m_s-X_s)\,ds+L\,dB_s+\sigma_0^{\mathrm{vec}}\,dW_s^0,
\qquad X_s\in\mathbb{R}^2,
\]
with $LL^\top=C$, $C_{11}=\sigma_1^2$, $C_{22}=\sigma_2^2$, $C_{12}=\rho\,\sigma_1\sigma_2$, and parameters $a=1$, $\sigma_1=0.4$, $\sigma_2=0.3$, $\rho=0.5$, $\sigma_0^{\mathrm{vec}}=(0.3,-0.2)$, and $T=1$. For a \rev{scenario dependent} initial law $\mathcal N_2(m_0,\operatorname{diag}(s_0^2))$, the conditional terminal law is exactly
\[
\mathcal N_2\!\left(
m_0+\sigma_0^{\mathrm{vec}}W_T^0,
\operatorname{diag}(s_0^2)e^{-2aT}
+\frac{1-e^{-2aT}}{2a}C
\right),
\]
with a non diagonal covariance. The common noise remains scalar, so the signature features are unchanged; this benchmark isolates the multidimensional measure argument.
\end{example}

\medskip
\noindent\emph{Scope of the benchmarks.}
The drifts in Examples~\ref{ex:ou-gaussian}--\ref{ex:ou-gmm} and
Example~\ref{ex:double-well} are respectively linear and superlinear in $x$
and are therefore unbounded. These models sit formally outside the boundedness
assumptions of Proposition~1, but they are retained because they provide
standard test cases with either exact conditional laws or \rev{high quality}
particle references.
The boundedness condition is inherited from the rough
path theorem used in Proposition~1 in the form stated in
\cite{friz2025mckean}. Extensions under standard linear growth and
dissipativity assumptions are expected, but they are not proved here. The diffusion coefficient in
Example~\ref{ex:multiplicative} is smooth, bounded, and elliptic. All six
examples interact through the conditional mean; richer dependence on the
conditional law is left for future work.

\subsection{Leakage free simulation and reference design}
\label{sec:reference-design}

For each main experiment, \(M=6000\) independent common noise scenarios
are generated for Examples~\ref{ex:ou-gaussian}--\ref{ex:ou-gmm} and
Example~\ref{ex:ou-2d}, and \(M=4000\) for
Examples~\ref{ex:double-well} and \ref{ex:multiplicative}. For each
scenario, a training system of \(N=200\)
interacting particles is simulated using its own empirical mean field.
Euler--Maruyama is used for all benchmarks except the \rev{double well} model,
whose superlinear drift is discretized using a tamed Euler scheme. The
benchmark dependent time grids contain 100 or 200 time steps. The scenarios
are split into \(70\%/15\%/15\%\) training, validation, and test sets.

For the models without a \rev{closed form} conditional law, we generate two
conditionally independent reference systems \(A\) and \(B\), each containing
\(N_{\mathrm{ref}}=2000\) particles and sharing only
\((\mu_0,W^0)\). Their initial samples and idiosyncratic Brownian motions are
independent, and each system evolves using its own empirical mean field.
Reference \(A\) supplies the Fourier integrals used by the
independent reference baseline, whereas reference \(B\) supplies the
regression targets. Hence, the measure inputs of this baseline and the
targets used to train \(\Psi\) are not constructed from the same finite
particle realization. The pooled cloud \(A\cup B\) is used only for
law level validation and density plots.

The stochastic resolution of the pooled reference is estimated from the
two replicates. Let \(D_{AB}\) denote the per scenario \(W_2\) distance
between their sorted empirical samples. Under an equal variance
independent error approximation, we use the mean distance estimate
\[
\varepsilon_{\mathrm{ref}}^{W_2}
\approx
\frac{1}{2}\,\mathbb E[D_{AB}].
\]

For the downstream functional, let \(Y_A\) and \(Y_B\) denote targets
constructed independently from the two reference systems. The relative
reference target floor is estimated by
\[
\varepsilon_{\mathrm{ref}}^{V}
=
\frac{
\left(\frac{1}{2}\mathbb E|Y_A-Y_B|^2\right)^{1/2}
}{
\left(\mathbb E|Y_B|^2\right)^{1/2}
}.
\]
We additionally report the noise adjusted diagnostic
\[
e_{\mathrm{adj}}
=
\sqrt{
\max\left\{
e_{\mathrm{obs}}^2
-
\left(\varepsilon_{\mathrm{ref}}^{V}\right)^2,
0
\right\}
}.
\]
This variance subtraction adjustment is used only for
Examples~\ref{ex:double-well} and \ref{ex:multiplicative}. It should be
interpreted as a diagnostic rather than as an exact deconvolution of the
finite reference error. Exact law benchmarks have no reference target
noise.

\subsection{Neural architecture and training}
\label{sec:architecture-numerical}

This subsection describes the numerical implementation of the factorized
approximation class introduced in Section~2. The first stage encodes the
initial distribution and the \rev{common noise path}, then trains a \rev{mixture density}
network to approximate the conditional law. The second stage freezes this law
representation and trains the cylindrical network $\Psi$ to evaluate the
target functional. We specify the input features, network dimensions,
optimization procedure, and evaluation functionals in turn.
Algorithm~\ref{alg:conditional-cylindrical-training} summarizes the complete
training and inference pipeline.

\begin{algorithm}[H]
\caption{Training and inference for the conditional cylindrical network}
\label{alg:conditional-cylindrical-training}
\begin{algorithmic}[1]
\Require Training scenarios
\(\{(\mu_0^i,W^{0,i},X_T^{i,1},\ldots,X_T^{i,N})\}_{i=1}^{N_{\rm sc}}\).
\Require Evaluation states and labels
\(\{(x^{i,r},Y^{i,r})\}_{i,r}\), where the labels
\(Y^{i,r}=V(x^{i,r},\mu_T^i)\)  are computed from the exact conditional law when available, and from an
independent reference system otherwise.
\Require Fourier order \(L\), signature level \(M_{\rm sig}\), and number of
mixture components \(J\), and the number of  epochs for $E_{\rm MDN}$ and  \(E_{\Psi}\).
\For{\(i=1,\ldots,N_{\rm sc}\)}
    \State Compute
    \(z_i=\bigl(E_L(\mu_0^i),
    S_{M_{\rm sig}}(r_T\widehat{\mathbf W}^{0,i})\bigr)\)
\EndFor
\State Fit the feature standardization map
\(\mathcal S_{\rm tr}\) using the training scenarios only

\For{\(i=1,\ldots,N_{\rm sc}\)}
    \State Set
    \[
    \widetilde z_i=\mathcal S_{\rm tr}(z_i)
    \]
\EndFor

\State Initialize the parameters \(\theta\) of the mixture-density network

\For{\(e=1,\ldots,E_{\rm MDN}\)}
    \State Shuffle the training scenarios and partition them into minibatches
    \For{each minibatch \(\mathcal B\) of scenarios}
        \State For every scenario \(i\in\mathcal B\), draw \(K_{\rm keep}\)
        terminal particles with indices \(\mathcal K_i\)
        \State Take one gradient step on the conditional negative log-likelihood
        \[
        \mathcal L_{\rm MDN}(\theta)
        =
        -\frac{1}{|\mathcal B|\,K_{\rm keep}}
        \sum_{i\in\mathcal B}\sum_{n\in\mathcal K_i}
        \log p_{\widehat N_\theta(\widetilde z_i)}\bigl(X_T^{i,n}\bigr)
        \]
    \EndFor
    \State Evaluate the validation NLL and retain the checkpoint
    with the lowest validation value
\EndFor

\State Restore the best MDN parameters and freeze \(\theta\)

\For{\(i=1,\ldots,N_{\rm sc}\)}
    \State Set \(\widehat\mu_T^i=\widehat N_\theta(\widetilde z_i)\)
    \State Compute
    \(c_i=\langle\varphi,\widehat\mu_T^i\rangle\) analytically from the
    Gaussian mixture
\EndFor
\State Initialize the parameters \(\eta\) of the cylindrical network
\(\Psi_\eta\)
\For{\(e=1,\ldots,E_{\Psi}\)}
     \State Shuffle the training evaluation pairs $(i,r)$ and iterate over minibatches $\mathcal{B}_{\Psi}$

    \State Update \(\eta\) by minimizing the mean-squared regression loss
    \[
    \mathcal L_{\Psi}(\eta)
    =
    \frac{1}{|\mathcal B_{\Psi}|}
    \sum_{(i,r)\in\mathcal B_{\Psi}}
    \left|
    Y^{i,r}
    -
    \Psi_\eta(x^{i,r},c_i)
    \right|^2
    \]

    \State Evaluate the validation regression loss and retain the checkpoint
    with the lowest validation value
\EndFor

\State Restore the best cylindrical-network parameters

\State At inference, compute
\Statex \hfill
\(
\displaystyle
z=
\left(
E_L(\mu_0),
S_{M_{\rm sig}}\!\left(r_T\widehat{\mathbf W}^{0}\right)
\right),
\qquad
\widetilde z=\mathcal S_{\rm tr}(z)
\qquad \text{and} \qquad
\widehat V(x,\mu_0,W^0)
=
\Psi_\eta\!\left(
x,
\left\langle
\varphi,
\widehat N_\theta(\widetilde z)
\right\rangle
\right)
\)
\hfill\mbox{}

\State \Return \(\widehat N_\theta\) and \(\Psi_\eta\)
\end{algorithmic}
\end{algorithm}

\paragraph{Features.}
The initial law encoding contains sine and cosine moments at
\(P=4\) positive frequencies
\(
k\in\{0.5,1,1.5,2\}
\)
in one dimension and at eight prescribed frequency vectors in two
dimensions. These moments are evaluated analytically for the three OU
initial law families using the corresponding characteristic functions.
The common noise encoding is the truncated signature of the
\rev{piecewise linear} interpolation of
$s\longmapsto(s,W_s^0)$,
computed at level \(2\) for Examples~\ref{ex:ou-gaussian}--\ref{ex:ou-gmm}
and Example~\ref{ex:ou-2d}, and at level \(3\) for
Examples~\ref{ex:double-well} and \ref{ex:multiplicative}. All inputs are
standardized using statistics computed from the training split.

\paragraph{Conditional law network.}
The conditional law estimator is a \rev{mixture density} network with a SiLU
multilayer perceptron trunk containing two hidden layers of width \(128\).
Its output parametrizes a Gaussian mixture with diagonal covariances,
using \(J=5\) components in one dimension and \(J=6\) in two dimensions.
The MDN is trained by conditional negative log likelihood for \(400\)
epochs using Adam, cosine learning rate annealing from \(10^{-3}\) to
\(10^{-5}\), a batch size of \(256\), and \rev{validation based} model selection.
At each gradient step, \(K_{\mathrm{keep}}=64\) particles are subsampled
from the \(N=200\) particles available in each scenario. Training histories
record the training and validation NLL together with a \rev{fixed subset}
validation conditional law discrepancy.

\paragraph{Cylindrical network.}
The downstream network \(\Psi\) is a two hidden layer, width \(128\) SiLU
MLP. Its inputs are the state \(x\) and the Fourier integrals of one of
three measure representations: the empirical \(N\) particle cloud; the
Gaussian mixture predicted by the MDN, whose Fourier integrals are
available analytically; or the exact conditional law when available and
the independent reference \(A\) representation otherwise. Each measure
channel has its own copy of \(\Psi\), trained on the same evaluation points.
In Examples~\ref{ex:double-well} and \ref{ex:multiplicative}, reference \(A\)
supplies the measure inputs, while reference \(B\) supplies the corresponding
regression labels. The MDN is frozen during this second
stage, and each copy of \(\Psi\) is trained by least squares for \(350\)
epochs.

\paragraph{Target functionals.}
The moment functional is
\[
V_{\mathrm{mom}}(x,\mu)
=
\langle x,m_1(\mu)\rangle
+
\operatorname{tr}\!\bigl(\operatorname{Cov}(\mu)\bigr)
+
\sin(x_1),
\qquad
m_1(\mu):=\int_{\mathbb R^d}y\,\mu(dy).
\]
For \(d=1\), we also consider
\[
V_{\mathrm{quant}}(x,\mu)
=
x\bigl(q_{0.75}(\mu)-q_{0.25}(\mu)\bigr)
+
F_\mu(0)
+
\sin x,
\qquad
F_\mu(0):=\mu((-\infty,0]).
\]
The second functional is a numerical stress test beyond the global
continuity assumptions of Theorem~1, since quantiles and
\(\mu\mapsto F_\mu(0)\) may be discontinuous at measures with atoms.

\paragraph{Metrics.}
In one dimension, \(\mathcal W_2\) is computed from \(2048\) midpoint
quantiles, with MDN quantiles obtained by numerical bisection. In two
dimensions, we use sliced \(\mathcal W_2\) over \(16\) deterministic
half circle directions. Downstream accuracy is measured by the relative
\(L^2\) test error over scenarios and evaluation points. Unless stated
otherwise, displayed dispersions are standard deviations across five
independent seeds, with both the simulated data and network initialization
regenerated.

Figure~\ref{fig:training-convergence} shows the optimization diagnostics
for a representative OU--GMM run. The MDN training and validation NLL and
the validation conditional law discrepancy decrease rapidly before
stabilizing, whereas the downstream relative \(L^2\) error improves more
gradually. This suggests using separate validation criteria for the
conditional law and functional regression stages.

\begin{figure}[H]
    \centering
    \includegraphics[width=0.98\linewidth]{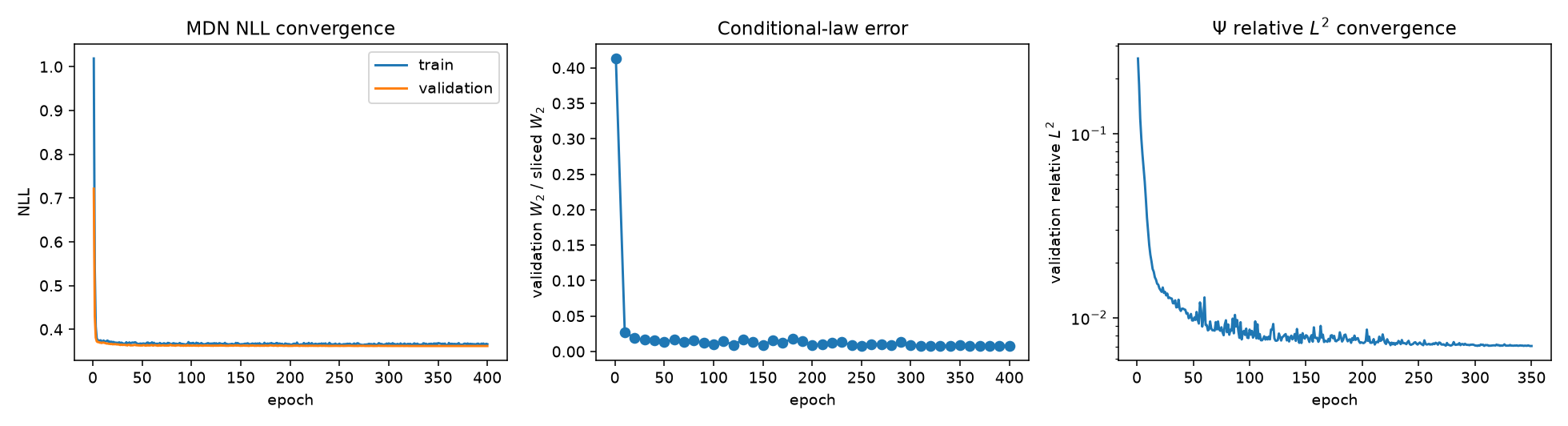}
    \caption{Representative optimization curves for the OU--GMM benchmark
    at seed \(0\): MDN training and validation NLL, validation
    conditional law discrepancy, and downstream \(\Psi\) relative
    \(L^2\) error.}
    \label{fig:training-convergence}
\end{figure}

\subsection{Main conditional law results}
\label{sec:law-results}
Table~\ref{tab:law-main} reports the five seed headline results. The MDN discrepancy is below that of the empirical $N=200$ cloud for every benchmark. The \rev{non Gaussian} OU models retain the same advantage: the mean $\mathcal W_2$ is $0.0054$ for the uniform initial law and $0.0082$ for the \rev{two component} mixture, compared with empirical cloud discrepancies $0.0450$ and $0.0550$.

\paragraph{Interpretation of the empirical comparison.}
The comparison with the empirical \(N=200\) particle representation is
amortized rather than matched at an identical information budget. The MDN
learns across many common noise scenarios and can exploit structure shared
across scenarios, whereas the empirical particle measure is constructed
independently within each test scenario and receives no cross scenario
training. The comparison therefore quantifies the gain from amortized
conditional law estimation and should not be interpreted as a comparison
between two estimators trained from exactly the same information.

\begin{table}[H]
\centering
\caption{Conditional law approximation over five independent seeds.
Values are means \(\pm\) standard deviations across seeds.
The reference resolution is reported only for benchmarks without a
\rev{closed form} conditional law. The \rev{one dimensional} benchmarks use
\(\mathcal W_2\), whereas Example~\ref{ex:ou-2d} uses sliced
\(\mathcal W_2\).}
\label{tab:law-main}
\small
\begin{tabular}{lcccc}
\toprule
Benchmark
& Validation NLL
& MDN discrepancy
& Empirical \(N=200\)
& Reference resolution \\
\midrule
OU--Gaussian
& \(0.2951\pm0.0045\)
& \(0.0057\pm0.0008\)
& \(0.0511\pm0.0008\)
& -- \\

OU--uniform
& \(0.2187\pm0.0033\)
& \(0.0054\pm0.0008\)
& \(0.0450\pm0.0005\)
& -- \\

OU--GMM
& \(0.3642\pm0.0023\)
& \(0.0082\pm0.0009\)
& \(0.0550\pm0.0008\)
& -- \\

Double well
& \(0.8373\pm0.0064\)
& \(0.0240\pm0.0011\)
& \(0.0817\pm0.0015\)
& \(0.0179\pm0.0003\) \\

Multiplicative
& \(0.0232\pm0.0086\)
& \(0.0127\pm0.0006\)
& \(0.0403\pm0.0012\)
& \(0.0092\pm0.0001\) \\

OU--2D
& \(0.1977\pm0.0072\)
& \(0.0064\pm0.0003\)
& \(0.0448\pm0.0002\)
& -- \\
\bottomrule
\end{tabular}
\end{table}

The Gaussian OU differential entropy floor is computed scenario by scenario,
\[
\frac1M\sum_{i=1}^{M}\frac12\log\bigl(2\pi e\,v_T^{(i)}\bigr),
\]
rather than from an aggregated variance. It equals $0.2896\pm0.0012$ across seeds ($0.2901$ at seed 0), while the validation NLL is $0.2951\pm0.0045$. The network is therefore close to, but not exactly at, the entropy floor. For Example~\ref{ex:ou-2d}, the exact entropy is $0.1800\pm0.0026$, leaving a larger gap that is consistent with representing a correlated Gaussian through diagonal mixture components. For the particle reference benchmarks, the MDN errors are close to but above the measured pooled reference resolutions, showing that finite reference uncertainty remains non negligible.

\begin{figure}[H]
\centering
\includegraphics[width=0.96\linewidth]{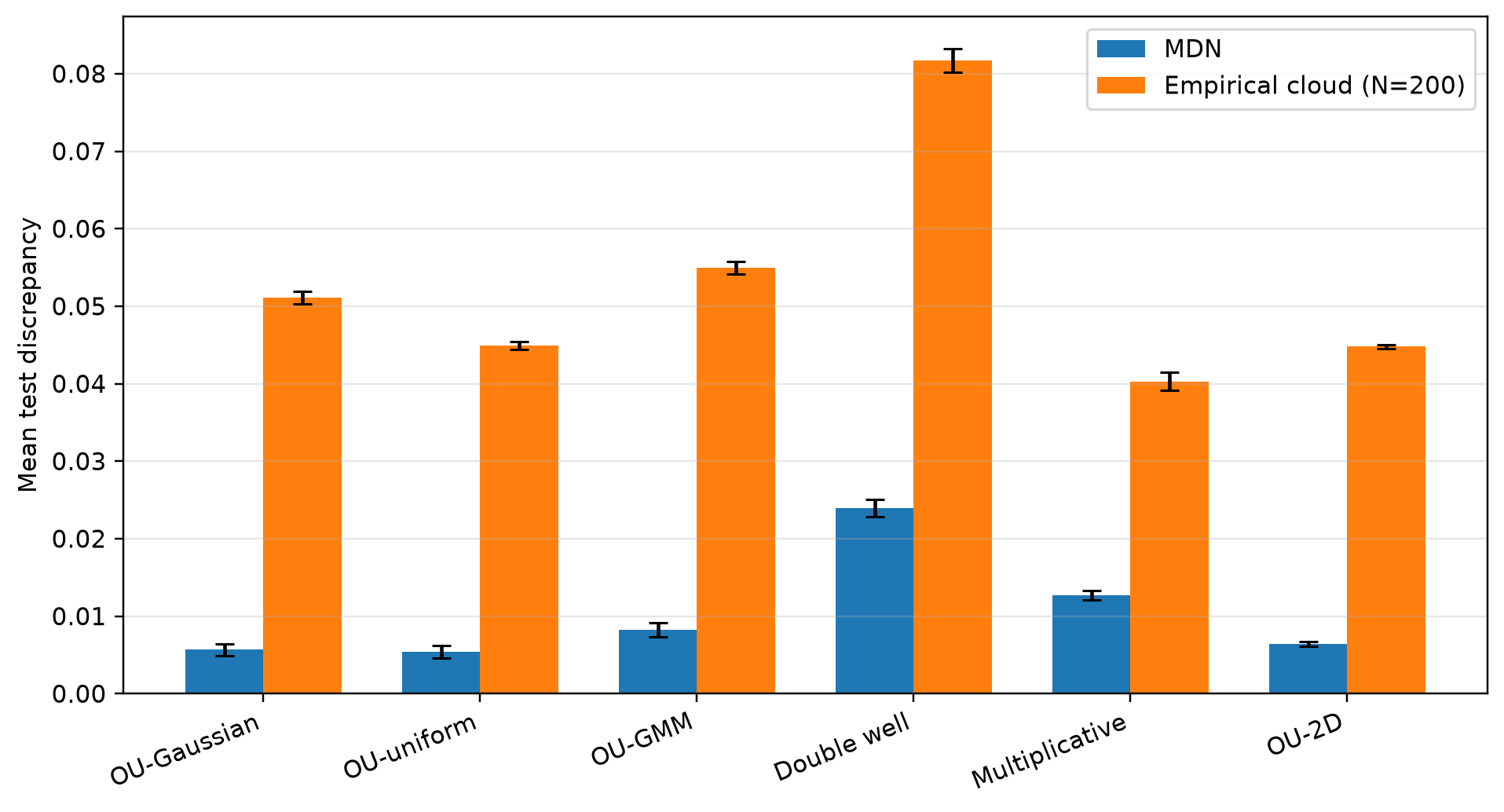}
\caption{Mean test law discrepancy for the MDN and the empirical $N=200$ particle cloud. Error bars are standard deviations across five independent seeds.}
\label{fig:law-errors}
\end{figure}

\rev{Figures~\ref{fig:double-well-law} and \ref{fig:ou-2d-law} provide
representative scenario level comparisons. The first shows that the MDN
captures asymmetric and bimodal conditional densities in the nonlinear
double well model. The second shows how several diagonal Gaussian components
combine to reproduce the location, dispersion, and correlation of the exact
two dimensional conditional law.}

\rev{For the double well model, the three displayed scenarios illustrate
substantially different conditional shapes. Scenario 94 concentrates most of
its mass in the left well while retaining a smaller right component. Scenario
148 displays two clearly separated modes of comparable importance, whereas
scenario 74 is dominated by the right well with a broad secondary mode on the
left. In every case, the MDN follows the locations, relative heights, and
widths of the reference modes despite the sampling fluctuations visible in
the particle histograms.}

\begin{figure}[H]
\centering
\includegraphics[width=0.98\linewidth]{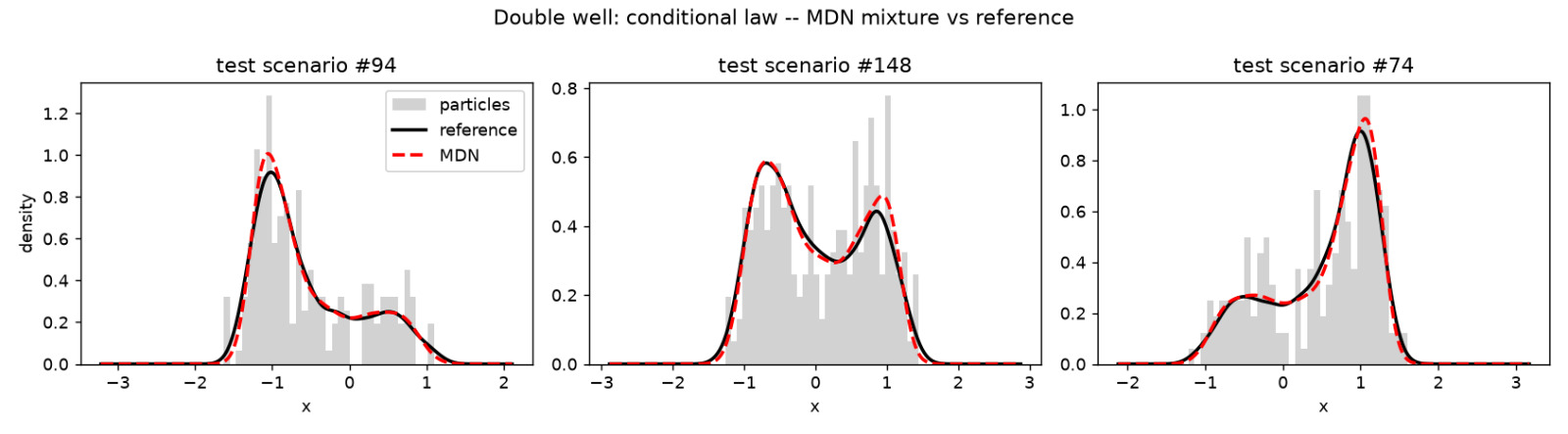}
\caption{\rev{Double-well benchmark (seed 0): predicted Gaussian-mixture densities $(J=5)$ against the N=200 terminal-particle histograms and kernel density estimates of the pooled independent reference systems ($N_{ref}=2000$ particles per replicate), for three representative test scenarios. The common noise tilts the occupancy of the two wells.}}
\label{fig:double-well-law}
\end{figure}

\begin{figure}[H]
\centering
\includegraphics[width=0.98\linewidth]{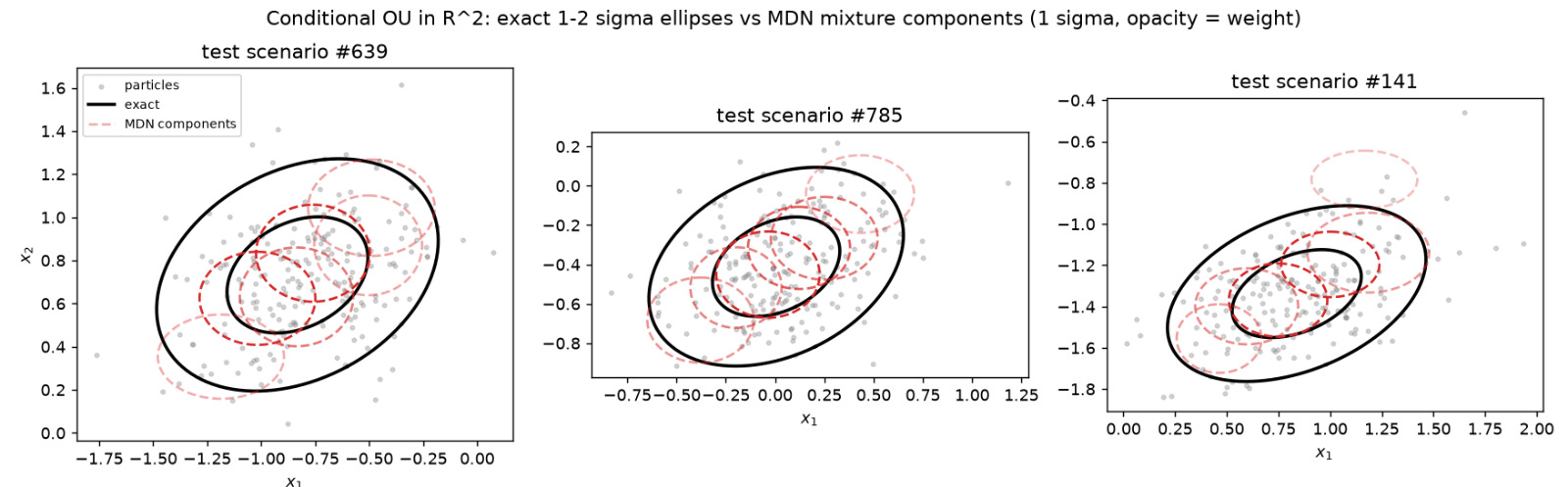}
\caption{\rev{Two-dimensional conditional OU benchmark (seed~0): exact conditional Gaussian
$1\sigma$ and $2\sigma$ ellipses (black), $N=200$ terminal particles (gray),
and $1\sigma$ ellipses of the $J=6$ diagonal Gaussian components predicted by
the MDN (red dashed, with opacity proportional to component weight), for three
representative test scenarios. The exact covariance is non-diagonal; the MDN captures the resulting correlation by combining several
diagonal components.}}
\label{fig:ou-2d-law}
\end{figure}

\subsection{Functional approximation}
\label{sec:functional-results}
The previous experiments provide numerical evidence consistent with the approximation architecture covered by Theorem~\ref{thm:main}.
We now consider a complementary stress test based on empirical quantiles.
Although quantiles generally fall outside the continuity assumptions of
Theorem~\ref{thm:main}, they are of considerable practical importance in
uncertainty quantification and risk analysis. The purpose of this experiment
is therefore not to validate the theorem, but to investigate the empirical
robustness of the proposed architecture on nonsmooth distributional
functionals.

Table~\ref{tab:functional-main} shows that the MDN measure channel improves the moment functional approximation on all six benchmarks. For exact law benchmarks, the final column is a genuine oracle representation. For double\_well and multiplicative it is an \emph{independent reference} representation: its Fourier inputs come from reference $A$ and its labels from reference $B$. It is therefore not a deterministic lower bound and can be slightly worse than the learned MDN after amortization.

\begin{table}[H]
\centering
\caption{Moment functional: relative $L^2$ test error in percent over five independent seeds. ``Exact/indep.'' denotes exact Fourier integrals for \rev{closed form} benchmarks and \rev{leakage free} independent reference integrals otherwise.}
\label{tab:functional-main}
\small
\begin{tabular}{lccccc}
\toprule
Benchmark & Empirical & MDN & Exact/indep. & Target floor & MDN adjusted \\
\midrule
OU--Gaussian   & $3.78\pm0.13$ & $0.49\pm0.07$ & $0.16\pm0.02$ & -- & -- \\
OU--uniform    & $3.18\pm0.14$ & $0.48\pm0.11$ & $0.13\pm0.02$ & -- & -- \\
OU--GMM        & $4.22\pm0.17$ & $0.70\pm0.07$ & $0.18\pm0.02$ & -- & -- \\
Double well    & $5.31\pm0.15$ & $2.24\pm0.15$ & $2.47\pm0.04$ & $1.77\pm0.03$ & $1.36\pm0.21$ \\
Multiplicative & $2.62\pm0.15$ & $1.40\pm0.41$ & $1.25\pm0.13$ & $0.79\pm0.04$ & $1.11\pm0.54$ \\
OU--2D         & $3.11\pm0.07$ & $0.69\pm0.06$ & $0.35\pm0.06$ & -- & -- \\
\bottomrule
\end{tabular}
\end{table}

For the nonlinear rows, the observed error includes the finite noise of the independent regression targets. The \rev{double well} MDN error is $2.24\%$. After subtracting the estimated target noise in quadrature, the error is $1.36\%$. The corresponding multiplicative values are $1.40\%$ observed and $1.11\%$ adjusted. These adjusted quantities are diagnostic rather than exact deconvolutions, but they make the remaining finite reference uncertainty explicit.

\begin{figure}[H]
\centering
\includegraphics[width=0.96\linewidth]{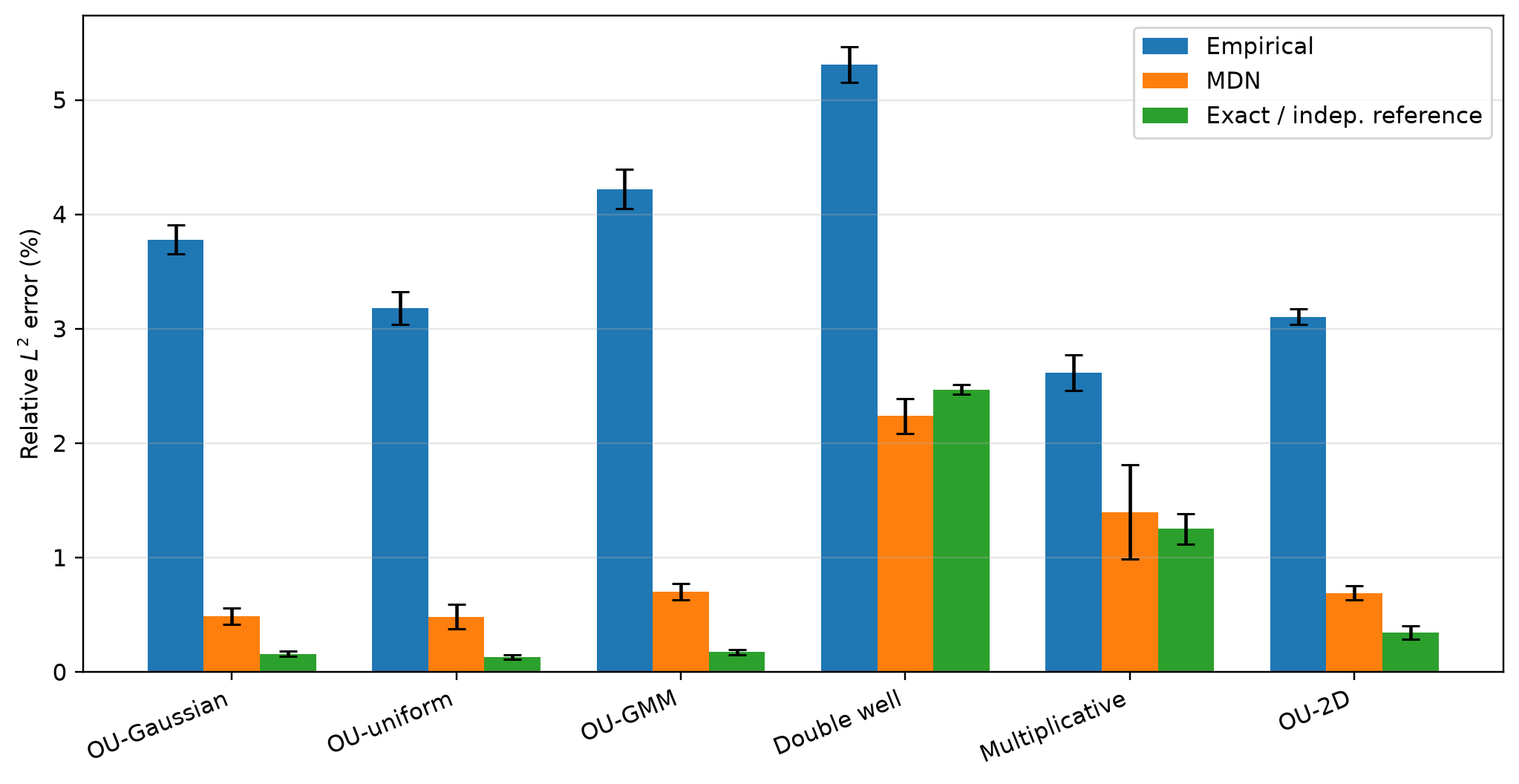}
\caption{Moment functional relative $L^2$ error for empirical, MDN and exact/independent reference measure channels. Error bars are standard deviations across five seeds.}
\label{fig:functional-errors}
\end{figure}

The quantile and tail stress test is summarized in Table~\ref{tab:quantile-main}. It is reported for seed 0 because the five seed campaign was reserved for the headline moment functional. The MDN remains substantially more accurate than the empirical plug in for both \rev{non Gaussian} initial law families and for the nonlinear models.

\begin{table}[H]
\centering
\caption{Quantile and tail functional at seed 0: relative $L^2$ test error in percent. Parenthesized values are noise adjusted for finite reference targets.}
\label{tab:quantile-main}
\small
\begin{tabular}{lccc}
\toprule
Benchmark & Empirical & MDN & Exact/independent reference \\
\midrule
OU--Gaussian   & $5.02$ & $0.82$ & $0.45$ \\
OU--uniform    & $4.12$ & $0.84$ & $0.57$ \\
OU--GMM        & $5.07$ & $1.34$ & $1.24$ \\
Double well    & $6.27$ & $2.74$ ($1.72$) & $3.16$ ($2.33$) \\
Multiplicative & $4.55$ & $2.32$ ($1.59$) & $2.42$ ($1.73$) \\
\bottomrule
\end{tabular}
\end{table}

\subsection{Sensitivity to features and neural hyperparameters}
\label{sec:sweeps}
The sweeps are one factor at a time experiments with all other settings fixed. Table~\ref{tab:sweeps} reports means over three seeds. They are empirical sensitivity checks, not approximation rate estimates.

\begin{table}[H]
\centering
\caption{Hyperparameter sensitivity. Values in each result sequence correspond to the settings listed in the middle column.}
\label{tab:sweeps}
\scriptsize
\begin{tabular}{p{0.23\linewidth}p{0.28\linewidth}p{0.40\linewidth}}
\toprule
Sweep and benchmark & Settings & Test result \\
\midrule
Mixture size $J$, OU--GMM & $1,2,3,5,8$ & mean $W_2$: $0.0199,0.0086,0.0084,0.0083,0.0086$ \\
Fourier frequencies $P$, OU--GMM & $1,2,4,6,8$ & mean $W_2$: $0.0163,0.0076,0.0083,0.0087,0.0094$ \\
Fourier frequencies $P$, OU--uniform & $1,2,4,6,8$ & mean $W_2$: $0.0049,0.0049,0.0054,0.0055,0.0061$ \\
Learning rate, OU--GMM & $10^{-4},3\cdot10^{-4},10^{-3},3\cdot10^{-3}$ & mean $W_2$: $0.0095,0.0079,0.0083,0.0084$ \\
MDN depth, OU--GMM & $1,2,3,4$ & mean $W_2$: $0.0080,0.0083,0.0075,0.0076$ \\
$\Psi$ depth, OU--GMM & $1,2,3,4$ & relative $L^2$ (\%): $0.79,0.75,0.74,0.73$ \\
Signature level, \rev{double well} & $1,2,3,4,5,6$ & mean $W_2$: $0.0721,0.0363,0.0238,0.0220,0.0251,0.0298$ \\
\bottomrule
\end{tabular}
\end{table}

The mixture sweep gives a structural sanity check: one Gaussian component cannot represent the \rev{two component} terminal family, while $J=2$ removes most of the gap and larger mixtures yield no systematic gain. For the sampled OU--GMM family, two positive Fourier frequencies already
provide the lowest mean discrepancy among the tested configurations. Additional frequencies produce no systematic improvement. The learning rate and depth sweeps show a broad stable region rather than a sharply tuned architecture. By contrast, signature depth has a strong effect up to level 4 on the double well: level 1 contains only endpoint increments, while higher levels incorporate increasingly rich path information. The law discrepancy decreases substantially through level 4, after which the improvement saturates and the error slightly increases, suggesting that a moderate signature truncation already captures most of the relevant path dependence.

\begin{figure}[H]
\centering
\includegraphics[width=0.73\linewidth]{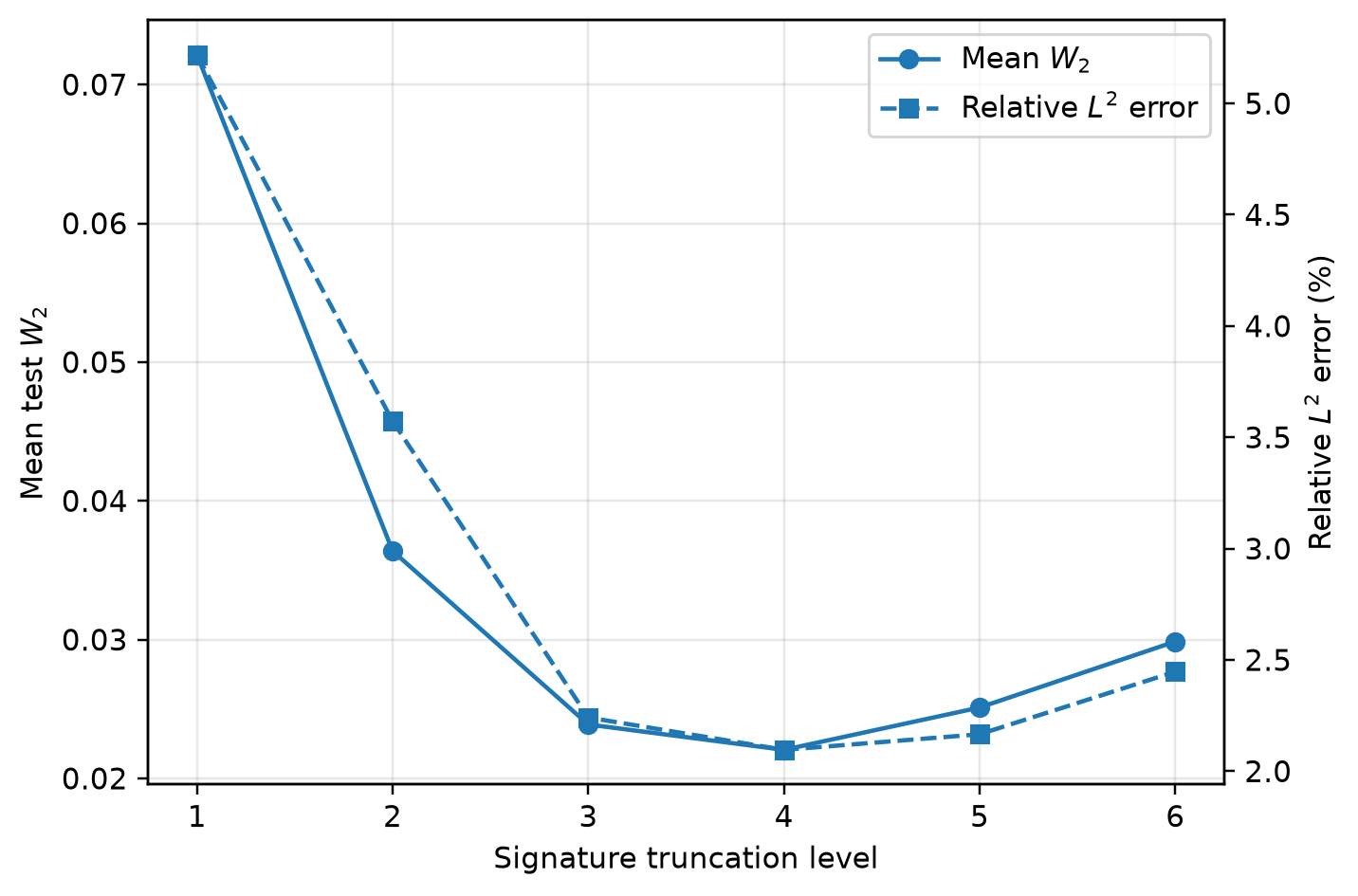}
\caption{Double well sensitivity to signature truncation. Both law and functional errors decrease materially from levels 1 to 4.}
\label{fig:signature-sweep}
\end{figure}

\begin{figure}[H]
\centering
\includegraphics[width=0.73\linewidth]{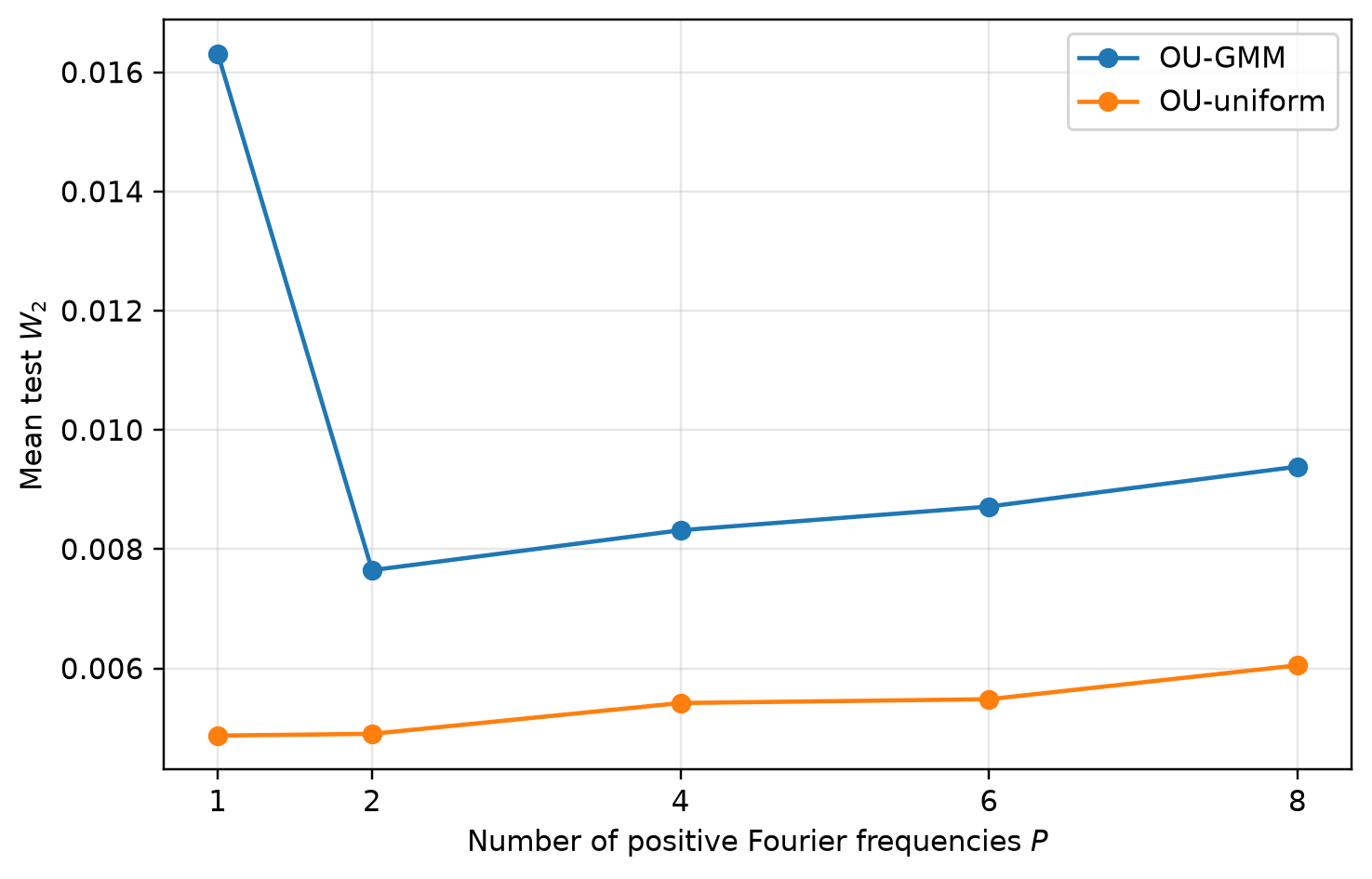}
\caption{Fourier order sensitivity for the two \rev{non Gaussian} OU benchmarks. The GMM family needs more than one positive frequency, whereas the uniform family is already stable at low order.}
\label{fig:fourier-sweep}
\end{figure}

\subsection{Single terminal observation per common noise scenario}

\label{sec:single-observation}

In this experiment, the particle system is still simulated with \(N=200\)
interacting particles, but the MDN receives only one fixed terminal
observation from each common noise scenario during training and validation.
Across three seeds on the Gaussian OU benchmark, it reaches
$W_2=0.0362\pm0.0054$, below the $N=200$ empirical cloud discrepancy
$0.0514\pm0.0010$. The downstream moment functional error is
$2.92\%\pm0.34\%$, compared with $3.74\%\pm0.16\%$ for the empirical plug in.
The full information configuration remains much more accurate, so the result
supports feasibility rather than equivalence.

A separate seed 0 sweep gives
\[
\begin{array}{c|cccc}
K_{\rm keep} & 1 & 8 & 64 & 200 \\
\hline
\text{mean }W_2 & 0.0069 & 0.0060 & 0.0056 & 0.0058
\end{array}
\]
Resampling particles across epochs makes the final estimate relatively insensitive to $K_{\rm keep}$, while larger values reduce gradient variance.

\subsection{It\^o--Stratonovich consistency}
\label{sec:ito-strat-num}
Using the same Brownian increments, $M=300$ scenarios, $N=2000$ particles and 200 time steps, we compare: Euler--Maruyama for the target It\^o equation; a Heun scheme for the Stratonovich equation with the corrected drift $a(m_s-x)-\frac12\sigma_0\sigma_0'(x)$; and the same Heun scheme with the naive uncorrected drift. The corrected gap to the It\^o reference is $0.00449\pm0.00348$, below the empirical resolution $0.01360$. The naive gap is $0.04941\pm0.02048$, approximately one order of magnitude larger. Thus the \rev{coefficient level} correction is numerically visible, while the paper's signature pipeline remains consistent with using the geometric lift only as a feature representation.

\bibliographystyle{plainurl}
\bibliography{AHR_references}

\end{document}